\documentclass[10pt,a4paper]{amsart}

\makeatletter
\renewcommand{\subsection}{%
  \@startsection{subsection}{2}{\z@}%
  {2ex \@plus .5ex \@minus .2ex}%
  {.8ex}%
  {\normalfont\bfseries\raggedright}%
}
\makeatother

\title{Complete $(2+1)$-dimensional Ricci flow spacetimes}
\author{Luke T. Peachey}
\address{Freie Universität Berlin, 14195 Berlin, Deutschland}
\email{l.peachey@fu-berlin.de}
\urladdr{www.lukepeachey.com}
\date{}

\usepackage[backend=biber,style=alphabetic]{biblatex}
\usepackage{amsfonts}
\usepackage{amsmath}
\usepackage{amssymb}
\usepackage{amsthm}
\usepackage{appendix}
\usepackage{array}
\usepackage[UKenglish]{babel}
\usepackage{breqn}
\usepackage{csquotes}
\usepackage{enumitem}
\usepackage{faktor}
\usepackage{framed}
\usepackage{graphicx}
\usepackage{mathrsfs}
\usepackage{mathtools}
\usepackage{setspace}
\usepackage{subcaption}
\usepackage{thmtools, thm-restate}
\usepackage{tikz}
\usepackage{tikz-cd}
\usepackage{bm}

\usepackage{hyperref}
\hypersetup{pdftitle={Complete $(2+1)$-dimensional 
Ricci flow spacetimes},
}

\hypersetup{hidelinks}
\usepackage[margin=1.2in]{geometry}

\graphicspath{{./Images/}}

\newcommand{\abs}[1]{\lvert#1\rvert}
\newcommand{\norm}[1]{\lVert#1\rVert}

\newtheorem{thm}{Theorem}[section] 
\newtheorem{lem}[thm]{Lemma} 

\newtheorem*{claim}{Claim} 
\newtheorem{cor}[thm]{Corollary}

\theoremstyle{definition}
\newtheorem{defn}[thm]{Definition}
\newtheorem{eg}[thm]{Example}

\theoremstyle{remark} 
\newtheorem*{rem}{Remark}

\DeclareMathOperator\supp{supp}
\DeclareMathOperator\Ric{Ric} 

\def\XXint#1#2#3{{\setbox0=\hbox{$#1{#2#3}{\int}$ }
\vcenter{\hbox{$#2#3$ }}\kern-.6\wd0}}

\begin{document}

\maketitle

\begin{abstract}
For a complete $(2+1)$-dimensional Ricci flow spacetime, we show that every worldline persists until the terminal time of the flow, and that at each time, the collection of worldlines not yet present is scattered. Consequently, only countably many worldlines can have inception time strictly larger than the initial time of the flow. Under non-negative curvature, every spatial slice must be diffeomorphic to a fixed surface and the spacetime is isomorphic to a classical Ricci flow on this surface. As an application, for any two-dimensional Ricci soliton, the completeness of the metric forces completeness of its soliton vector field.
\end{abstract}

\section{Introduction}
Ricci flow spacetimes were introduced by Kleiner and Lott \cite{kleiner2017singular} to describe Ricci flows through singularities. As a framework they allow the study of Ricci flows in which the underlying spatial manifold is allowed to change with time. Throughout this paper, $I=(\alpha,\beta)$ denotes a non-empty open interval, where we allow $\alpha=-\infty$ and $\beta=\infty$.

\begin{defn}
A Ricci flow spacetime over $I$ is a tuple $(\mathcal{M}^{n+1} , \mathfrak{t} , \partial_t , g)$ where
\begin{itemize}
\item $\mathcal{M}$ is a smooth connected $(n+1)$-manifold without boundary;
\item $\mathfrak{t} : \mathcal{M} \twoheadrightarrow I$ is a smooth surjective submersion;
\item $\partial_t$ is a smooth vector field on $\mathcal{M}$ satisfying $\partial_t  \mathfrak{t}  \equiv 1$;
\item $g$ is a smooth inner product on $(T\mathcal{M})^{\mathrm{spat}} := \ker(d\mathfrak{t})$, whose restriction $g(t)$ to the spatial slice $\mathcal{M}_{t} = \mathfrak{t}^{-1}(t)$ is a Riemannian metric, for all $t \in I$;
\item $g$ and $\partial_t$ satisfy the Ricci flow spacetime equation $\mathcal{L}_{\partial_t} g  + 2 \Ric(g) =0$.
\end{itemize}
The spacetime is complete if $(\mathcal{M}_t,g(t))$ is complete for every $t \in I$.
\end{defn}
The purpose of this paper is to show that, in spatial dimension two, completeness severely restricts how the underlying manifold can change. For $x\in\mathcal M$, consider the maximal integral curve of $\partial_t$ through $x$, which we refer to as the \emph{worldline} through $x$. Let $q:\mathcal M \to \mathcal W$ be the quotient by worldlines, and equip $\mathcal W$ with the quotient topology. Define the inception and extinction times of a worldline $w \in \mathcal{W}$ by
$$\mathfrak t_-(w) :=
\inf_{x \in q^{-1}(w)} \mathfrak t(x),
\quad
\mathfrak t_+(w) := \sup_{x \in q^{-1}(w)} \mathfrak t(x),$$
respectively. In particular, the image of a worldline under $\mathfrak{t}$ is
$$ \mathfrak{t}(q^{-1}(w)) = \left( \mathfrak{t}_-(w) , \mathfrak{t}_+(w) \right) \subset I.$$

Recall that a topological space is scattered if every non-empty subspace has an isolated point.

\begin{restatable}{thm}{countable}\label{thm:main}
Let $(\mathcal M^{2+1},\mathfrak t,\partial_t,g)$ be a complete
Ricci flow spacetime over $I=(\alpha,\beta)$, and let $\mathcal{W}$ be its space of worldlines. Then every worldline has maximal extinction time 
$$\mathfrak t_+(w) = \beta, \quad \forall w \in \mathcal{W}.$$
Consequently, $\mathcal W$ is canonically a smooth surface. Moreover, for any $s \in I$, the collection of worldlines with inception time at least $s$
$$ \{ w \in \mathcal{W} : \mathfrak{t}_-(w) \geq s\},$$
is a scattered subset of $\mathcal{W}$. In particular
$$  \{ w \in \mathcal{W} : \mathfrak{t}_-(w) > \alpha\}$$
is at most countable.
\end{restatable}

The countability conclusion is sharp, as the following example shows.

\begin{eg}\label{eg1}
Fix $\tau>0$, $M:= S^1 \times S^1$ and choose a point $p \in M$. Consider the open subset $\mathcal{M} := \left(M \times (0,\infty)\right) \setminus \left( \{ p \} \times (0,\tau] \right)$. In other words, our spatial slices are the punctured torus $\mathcal{M}_t = M  \setminus \{ p \}$ for $t \in (0,\tau]$ and the whole torus $\mathcal{M}_t = M$ for $t > \tau$. Let $h$ denote the complete hyperbolic metric on $M \setminus \{p\}$, and set $g(t) = 2th$, the big-bang Ricci flow on $M \setminus \{p\}$, for $t \in (0,\tau]$. Applying Topping's contracting-cusp construction \cite[][Theorem 1.2]{topping2012uniqueness} to the metric $2\tau h$, we obtain a complete Ricci flow on $M$ for $t\in(\tau,\tau+\delta]$, for some $\delta>0$, which converges locally smoothly to $2\tau h$ on $M \setminus \{p\}$ as $t\downarrow\tau$. Since $M$ is a torus, the flow can then be continued for all $t>\tau$. In particular, when equipped with the standard time function $(x,t) \mapsto t$ and time vector field $\frac{\partial}{\partial t}$, we have a complete Ricci flow spacetime over $(0,\infty)$, where the worldline corresponding to $p$ has inception time $\tau>0$. Lifting the construction to the universal cover $\mathbb{C} \to M$, we find a complete Ricci flow spacetime over $(0,\infty)$ with a countably infinite number of worldlines whose inception time is $\tau>0$.
\end{eg}

Despite such examples, under non-negative curvature this behaviour is ruled out. To state the result we use the following notion of isomorphism.

\begin{defn}
Two Ricci flow spacetimes $(\mathcal{M}^{n+1},\mathfrak{t},\partial_t,g)$ and $(\mathcal{N}^{n+1},\mathfrak{s},\partial_s,G)$ over $I$ are \textit{isomorphic} if there exists a smooth diffeomorphism $\Phi : \mathcal{M} \rightarrow \mathcal{N}$ such that
\begin{equation*}
 \mathfrak{t} = \mathfrak{s} \circ \Phi, \quad \Phi_*(\partial_t) = \partial_s, \quad g = \Phi^*G.   
\end{equation*}
\end{defn}

\begin{restatable}{thm}{nn}\label{thm:nonnegative}
Let $(\mathcal M^{2+1},\mathfrak t,\partial_t,g)$ be a complete
Ricci flow spacetime over $I=(\alpha,\beta)$ with non-negative curvature $K_g\geq0$. Then 
$$\mathfrak t_-(w) = \alpha, \quad \forall w \in \mathcal{W}.$$
Consequently, the spacetime is isomorphic to a complete Ricci flow on a fixed smooth surface.
\end{restatable}

Our main theorem has a nice application to the study of Ricci solitons. Arising from the early work of Hamilton \cite{hamilton1988ricci,hamilton1993harnack} (see also \cite{chow2006hamilton}), they arise naturally in the study of self-similar solutions to the Ricci flow. When the soliton vector field is complete, the soliton generates a self-similar solution to the Ricci flow. For surfaces, we show that completeness of the metric forces completeness of the soliton vector field.

\begin{restatable}{cor}{sol}\label{cor:sol}
Suppose $(M^2,g,X)$ is a Ricci soliton
\begin{equation*}
    2\Ric(g) + \varepsilon g = \mathcal{L}_Xg,
\end{equation*}
for some $\varepsilon \in \mathbb{R}$. If $g$ is complete then the vector field $X$ is complete.
\end{restatable}

\subsection{Outline of the proof of Theorem~\ref{thm:main}}
A spacetime version of Chen's scalar-curvature estimate yields a length distortion estimate in dimension two. Together with completeness, this shows that every worldline persists until the terminal time. The space of worldlines can therefore be identified with a smooth surface $M$, with the spatial slices forming an increasing family of open subsets $M_t\subset M$.

For $s\in I$, consider the future domain
$$M_s^+:=\operatorname{Int}\bigcap_{t>s}M_t.$$
The area measures of the Ricci flow admit a weak limit $\mu_s^+$ on $M_s^+$ as $t\downarrow s$. We prove that both the spatial slice $M_s$ and the future domain $M_s^+$ are connected, and that their difference $M_s^+\setminus M_s$ has zero area with respect to any smooth Riemannian metric on $M$. 

On hyperbolic spatial slices, lifting to the unit disk allows us to compare the spacetime solution with suitable Ricci flows on the disk. Completeness forces the spacetime conformal factor to be proper in a suitable sense, so that the usual maximum principle can be applied despite the changing spatial domain.

If the set of missing worldlines contains a non-empty perfect subset, we localise to a compact part of it, place a non-atomic Radon measure there, and evolve this measure under the Ricci flow. Large parabolic rescalings converge smoothly away from the support to an expanding hyperbolic flow, giving the initial ordering needed to compare these flows with the spacetime solution. Meanwhile, every neighbourhood meeting its support acquires arbitrarily large area under these rescalings. 

At the endpoint $\tau$ of the hyperbolic comparison interval, we show that the perfect set meets $M_\tau^+$. If it already meets $M_\tau$, comparison gives a finite area bound. If not, it meets $M_\tau^+\setminus M_\tau$, and the same bound comes from the inception measure $\mu_\tau^+$ together with $M_\tau^+\setminus M_\tau$ having zero area. This contradicts the area blow-up. Therefore the missing set contains no perfect subset, and the Cantor--Bendixson theorem implies that it is scattered.

\vskip0.2cm

{\it Acknowledgements:} The author is grateful to Peter Topping, under whose supervision this project began, for his ideas and encouragement, and to Arjun Sobnack for many insightful discussions. Much of this work was carried out during the author's PhD, supported by the EPSRC (reference number 2104917, grant number EP/R513374/1).

\vskip0.2cm

{\it Declarations:} No generative artificial intelligence has been used in the production of this work.

\section{Worldlines and Extinction}\label{section 2}
Given a Ricci flow spacetime $(\mathcal{M},\mathfrak{t},\partial_t,g)$, since $\partial_t$ is a smooth vector field on $\mathcal{M}$, we may consider its maximal flow $\Phi$. The following are standard properties of the flow (see e.g. \cite{lee2013smooth}). The maximal flow domain
$$\mathcal{D} := \{ (x,\tau) \in \mathcal{M} \times \mathbb{R} : \mathfrak{t}(x) + \tau \in \mathfrak{t}( q^{-1}(q(x)))\}$$
is an open subset of $\mathcal{M} \times \mathbb{R}$, where $q:\mathcal{M} \to \mathcal{W}$ is the worldline quotient introduced above. For fixed $\tau \in \mathbb{R}$,
$$\mathcal{M}^{(\tau)} := \{ x \in \mathcal{M} : (x,\tau) \in \mathcal{D} \},$$
is an open subset of $\mathcal{M}$, and $\Phi_{\tau}: \mathcal{M}^{(\tau)} \rightarrow \mathcal{M}^{(-\tau)}$ is a smooth diffeomorphism. For any $s,t \in I$, consider those points at time $s$ which can flow until time $t$
$$\mathcal{M}_{s,t} = \mathcal{M}_s \cap \mathcal{M}^{(t-s)},$$
with corresponding flow
$$ \Phi_{s,t} := \Phi_{t-s}\vert_{\mathcal{M}_{s,t}} : \mathcal{M}_{s,t} \to \mathcal{M}_{t,s}.$$

Thus $\mathcal{M}_{s,t}$ is open in $\mathcal{M}_s$ and $\Phi_{s,t}$ is a smooth diffeomorphism. For any $t \in I$, if we define the continuous function $F_t :\mathcal{M} \to \mathcal{M} \times \mathbb{R}$ by $F_t(x) = (x,t - \mathfrak{t}(x))$, then
$$ \bigcup_{s \in I} \mathcal{M}_{s,t} = F_t^{-1} \mathcal{D},$$
is an open subset of $\mathcal{M}$. Set
\begin{equation*}
\mathcal{W}_t
    :=
    \{w\in\mathcal W:t \in \mathfrak t(q^{-1}(w))\}, \quad \forall t \in I.
\end{equation*}

\begin{lem}\label{lem: worldline}
Let $(\mathcal M^{n+1},\mathfrak t,\partial_t,g)$ be a Ricci flow
spacetime over $I$. The following properties hold.
\begin{enumerate}[label=(\roman*)]
\item\label{q open} The quotient map $q : \mathcal{M} \to \mathcal{W}$ is open.
\item\label{qt homeo} $\mathcal W_t \subseteq \mathcal{W}$ is open for every $t \in I$, with
$$q_t:=q\vert_{\mathcal M_t}: \mathcal M_t\longrightarrow\mathcal W_t$$
a homeomorphism. In particular, every $\mathcal W_t$ is a smooth $n$-dimensional manifold.

\item\label{non H manifold} $\mathcal W$ is a connected, second countable, $\mathrm{T}_1$-topological space, endowed with a canonical smooth atlas (possibly non-Hausdorff) via the maps $q_t$.

\item\label{wl embed} The map
\begin{equation*}
    (q,\mathfrak t): \mathcal M\longrightarrow \mathcal{W} \times I
\end{equation*} 
    is a diffeomorphism onto an open subset of $\mathcal{W} \times I$ with the atlas from \ref{non H manifold}, with $$(q,\mathfrak{t})_*\partial_t=\frac{\partial}{\partial t}.$$
\end{enumerate}
\end{lem}

\begin{rem}
$\mathcal{W}$ is not necessarily a smooth $n$-dimensional manifold as it may not be Hausdorff; worldlines may fail to coexist in a common time-slice.  
\end{rem}

\begin{proof}

We first observe that $q$ is open. Indeed, for any open subset $U\subset\mathcal M$,
$$q^{-1}(q(U)) = \bigcup_{\tau \in \mathbb{R}} \Phi_{\tau}(U \cap \mathcal{M}^{(\tau)} ).$$

For \ref{qt homeo},  we first note that
$$ q^{-1}(\mathcal{W}_t) = \bigcup_{s \in I} \mathcal{M}_{s,t},$$
is open in $\mathcal{M}$, and so $\mathcal{W}_t \subseteq \mathcal{W}$ is open. Since every worldline meets
$\mathcal M_t$ at most once, $q_t$ is a continuous bijection. If $U\subset\mathcal M_t$ is open, then 
$$q^{-1}(q_t(U))= F_t^{-1} \left( \Phi^{-1}(U) \right),$$
is open in $\mathcal M$, and hence $q_t$ is a homeomorphism. In particular, $q_t$ equips $\mathcal W_t$ with the structure of a smooth $n$-dimensional manifold.

For \ref{non H manifold}, since $q$ is a continuous open surjection from the connected and second countable manifold $\mathcal{M}$, $\mathcal{W}$ is also connected and second countable. Let $w,w'\in\mathcal W$ be distinct and choose $t \in \mathfrak t(q^{-1}(w'))$. If $w \notin\mathcal W_t$, then $\mathcal W_t$ is an open neighbourhood of $w'$ not containing $w$. Conversely, if $w\in\mathcal W_t$, then, since $\mathcal W_t$ is
Hausdorff, there is an open neighbourhood of $w'$ in
$\mathcal W_t$ which does not contain $w$. Thus
$\mathcal W\setminus\{w\}$ is open, and $\mathcal{W}$ is a $\mathrm{T}_1$-topological space.

The homeomorphisms $q_t$ transport the smooth structure of $\mathcal{M}_t$ to $\mathcal W_t$. These structures agree on overlaps, giving $\mathcal{W}$ a smooth atlas. Indeed, 
$$q_t^{-1}\circ q_s=\Phi_{s,t},$$
wherever both sides are defined.

Finally, as we have shown $q_t$ is a diffeomorphism, $(q,\mathfrak{t})$ is a local diffeomorphism. It is also injective as distinct points in a worldline must have distinct values of $\mathfrak{t}$, and is thus a diffeomorphism onto its open image. As $q$ is constant on worldlines and $\partial_t \mathfrak{t} \equiv 1$, the formula $(q,\mathfrak{t})_* \partial_t = \frac{\partial}{\partial t}$ is immediate.
\end{proof}

We now need quantitative control on how distances change under the flow. The starting point is a spacetime version of Chen's lower scalar-curvature estimate (cf. \cite[][Proposition 2.1]{chen2009strong}).

Recall the following definition of a time-preserving path \cite[Definition 1.11]{kleiner2017singular}.

\begin{defn}
Let $(\mathcal{M} , \mathfrak{t} , \partial_t , g)$ be a Ricci flow spacetime over $I$. Suppose $J \subseteq I$ is an interval and $\eta: J \rightarrow \mathcal{M}$ a smooth path. We say that $\eta$ is \textit{time-preserving} if $\mathfrak{t} (\eta (t)) = t$, $\forall t \in J$.
\end{defn}

\begin{thm}\label{Chen 2.1}
Fix $n \geq 2$. For any $\delta \in (0,\frac{2}{n})$ there exists $C(\delta,n) \in (0,\infty)$ with the following property. Let $(\mathcal{M}^{n+1} , \mathfrak{t} , \partial_t , g)$ be a Ricci flow spacetime over $I$. Fix $[t_0,t_1] \subseteq I$ and $\eta: [t_0,t_1] \rightarrow \mathcal{M}$ a time-preserving path. Let $\pi : T\mathcal{M} \rightarrow (T\mathcal{M})^{\mathrm{spat}}$ denote the spatial projection, and set
$$ B_t(r) := B_{g(t)}(\eta(t),r), \quad v(t) := \left|\pi(\eta'(t))\right|_{g(t)}.$$
Suppose there are constants $r_0, K> 0$ and $A\geq 1 + 4(n-1)r_0^{-2}(t_1-t_0)$ such that
\begin{itemize}
\item $B_{t}(2Ar_0) \Subset \mathcal{M}_t$, for every $t \in [t_0,t_1]$;
\item $\Ric(g(t)) \leq (n-1)r_0^{-2}$ on $B_{t}(r_0)$, for every $t \in [t_0,t_1]$;
\item $\mathcal{R}_{g(t_0)} \geq -K$ on $B_{t_0}(2Ar_0)$;
\item $v(t) \leq r_0^{-1}$, for every $t \in [t_0,t_1]$.
\end{itemize}
Then for each $t \in [t_0,t_1]$ and $x \in B_{t}(Ar_0)$, we have
\begin{equation*}
\mathcal{R}_{g(t)}(x) \geq - \max \{ \frac{1}{(\frac{2}{n}-\delta)(t-t_0) + \frac{1}{K} }, \frac{C}{A^2r_0^2} \}.
\end{equation*}
\end{thm}
\begin{proof}\label{Chen proof}
Choose a smooth non-increasing function $\phi : \mathbb{R} \rightarrow [0,1]$ with $\phi \equiv 1$ on $(-\infty,\frac{7}{8}]$, $\phi \equiv 0$ on $[1,\infty)$, and
$$ \left| 2(\phi')^2- \phi \phi''\right| \leq C_0\phi^{3/2},$$
for some constant $C_0 < \infty$. For $x \in \mathfrak{t}^{-1}([t_0,t_1])$, define the functions
$$ d(x) := d_{g(\mathfrak{t}(x))}(\eta(\mathfrak{t}(x)),x), \quad \rho(x) :=  \frac{d(x) + 3(n-1)r_0^{-1}(\mathfrak{t}(x) - t_0)}{2Ar_0},$$
and the localised scalar curvature 
\begin{equation*}
 u : \mathfrak{t}^{-1}([t_0,t_1]) \rightarrow \mathbb{R}, \quad   u(x):= \phi\left(\rho(x)\right) \cdot \mathcal{R}_{g(\mathfrak{t}(x))}(x).
\end{equation*}

In the region where the distance function is smooth, a direct calculation yields
\begin{equation}\label{eqn chen 1}
\square u = \frac{\mathcal{R} \phi'}{2Ar_0} \cdot \left( \square d + 3(n-1)r_0^{-1} \right) - 2 \nabla \phi \cdot \nabla \mathcal{R} - \frac{\mathcal{R} \phi''}{4A^2r_0^2}  + 2\phi |\Ric|^2,
\end{equation}
where $\square := {\partial_t} - \Delta_g$ denotes the heat operator on $\mathcal{M}$.

Consider now the infimum of $u$ on each spatial slice
\begin{equation*}
   m : [t_0,t_1] \to \mathbb{R}, \quad m(s) := \inf_{x \in \mathcal{M}_s} u(x).
\end{equation*}
Suppose for some fixed $s \in (t_0,t_1)$ that $m(s)<0$. As $u|_{\mathcal M_s}$ is continuous and has compact support in $B_s(2Ar_0)$, the minimum is attained at some point $x_0 \in \mathcal{M}_{s}$.

\begin{claim}
On a small neighbourhood of $x_0$, we have the differential inequality
\begin{equation*}
    \mathcal{R} \phi' \cdot \left(  \square d_g + 3(n-1)r_0^{-1} \right) \geq 0,
\end{equation*}
in the barrier sense.
\end{claim}
\begin{proof}[Proof of Claim]
By continuity, for any point $y$ in a sufficiently small neighbourhood of $x_0$, $\mathcal{R}_{g}(y) < 0$. At such a point $y$, if $d(y) \leq r_0$, then $\rho(y) \leq \frac{1}{2}$, and $\phi' = 0$ locally. In this case the inequality holds trivially, so assume $d(y) > r_0$. We combine Perelman's distance estimate \cite[][Lemma 8.3]{perelman2002entropy}, together with the bound on the drift $v(t) \leq r_0^{-1}$ to yield
\begin{equation*}
    \square d \geq -\frac{5(n-1)}{3r_0} - v(t) \geq -\frac{5(n-1)}{3r_0} - \frac{1}{r_0},
\end{equation*}
in the barrier sense. Combined with $\phi' \leq 0$, the Claim follows.
\end{proof}

Combining the Claim with (\ref{eqn chen 1}), in a neighbourhood of $x_0$ we have the differential inequality
\begin{equation}
\square u \geq - 2 \nabla \phi \cdot \nabla \mathcal{R} - \frac{\mathcal{R} \phi''}{4A^2r_0^2} + \frac{2}{n} \phi \mathcal{R}^2,
\end{equation}
in the barrier sense. Furthermore, $\phi > 0$ near $x_0$, and so within a possibly smaller neighbourhood of $x_0$, smooth points of $d$ satisfy
\begin{equation*}
\nabla \mathcal{R} = \frac{\nabla u}{\phi}- \frac{\mathcal{R}}{\phi} \nabla \phi, \quad \abs{\nabla\phi}^2 = \frac{(\phi')^2}{4A^2 r_0^2}.
\end{equation*}
Therefore, on this neighbourhood of $x_0$, the differential inequality
\begin{equation}\label{eqn chen 3}
\square u \geq - \frac{2}{\phi} \nabla \phi \cdot \nabla u + \frac{1}{4A^2 r_0^2} \left(  \frac{2(\phi')^2}{\phi}  - \phi'' \right) \mathcal{R} + \frac{2}{n} \phi \mathcal{R}^2,
\end{equation}
holds in the barrier sense. Using Young's inequality
\begin{align*}
  \frac{1}{4A^2 r_0^2} \left(  \frac{2(\phi')^2}{\phi}  - \phi'' \right) \mathcal{R}  &\geq  - \frac{C_0}{4A^2 r_0^2} | \mathcal R | \phi^{\frac12}\\
  &\geq - \frac{\delta}{2} \phi \mathcal{R}^2 - \frac{\delta}{2}\left( \frac{C}{A^2 r_0^2} \right)^2,
\end{align*}
where $C$ now depends on $\delta$. Therefore, we can simplify equation (\ref{eqn chen 3}) to the differential inequality
\begin{equation*}
     \square u  \geq - \frac{2}{\phi} \nabla \phi \cdot \nabla u + \left(\frac{2}{n} - \frac{\delta}{2}\right) \phi \mathcal{R}^2 -  \frac{\delta}{2} \left(\frac{C}{A^2 r_0^2}\right)^2.
\end{equation*}
It follows that
\begin{align*}
    \liminf_{h \downarrow 0} \frac{m(s + h) - m(s)}{h} &\geq \left(\frac{2}{n} - \frac{\delta}{2}\right) \phi \mathcal{R}^2 -  \frac{\delta}{2} \left(\frac{C}{A^2r_0^2}\right)^2 \\
&\geq \left(\frac{2}{n} - \delta\right) m^2(s) + \frac{\delta}{2} \left(m(s)^2 - \left( \frac{C}{A^2 r_0^2} \right)^2 \right).
\end{align*}
comparing with the barriers $- \frac{C}{A^2 r_0^2}$ and $\frac{-1}{(\frac{2}{n} - \delta)(t-t_0) + K^{-1}}$, we obtain the lower bound
\begin{equation*}
m(t) \geq - \max \{ \frac{1}{(\frac{2}{n}-\delta)(t-t_0) + \frac{1}{K} }, \frac{C}{A^2r_0^2} \}.
\end{equation*}
Since $u = \mathcal{R}_{g(t)}$ on $B_t(Ar_0)$, the result follows.
\end{proof}

For a point $x\in\mathcal M$ the initial time relevant to the preceding estimate need not be its inception time
$\mathfrak{t}_-(q(x))$; one may be able to reach $x$ from an earlier time along a time-preserving path.

\begin{defn}\label{defn hindsight}
Let $(\mathcal{M} , \mathfrak{t} , \partial_t , g)$ be a Ricci flow spacetime over $I=(\alpha,\beta)$. Define the \textit{hindsight function} $\mathfrak{h} : \mathcal{M} \rightarrow [\alpha,\beta)$ by
\begin{equation*}
    \mathfrak{h}(x) := \inf \{ s \in I  : \exists \textnormal{ time-preserving path } \eta: [s,\mathfrak{t}(x)] \rightarrow \mathcal{M} \textnormal{ with } \eta (\mathfrak{t}(x)) =x \}.
\end{equation*}
\end{defn}

One immediately sees that the hindsight function is non-increasing along worldlines. Indeed, if
$q(y) = q(x)$ and $\mathfrak{t}(y) > \mathfrak{t}(x)$, then any
time-preserving path terminating at $x$ can be, after a harmless smooth reparameterisation, concatenated with the worldline segment from $x$ to $y$.

Moreover, the hindsight function restricted to any spatial slice $\mathfrak{h}\vert_{\mathcal{M}_t}$ is locally constant. To see this, suppose $x,y \in \mathcal{M}_\tau$ are joined by a smooth curve
$$\gamma:[0,1] \to \mathcal{M}_\tau, \quad \gamma(0)=x, \quad \gamma(1) = y,$$
and choose a time-preserving path $\eta:[s,\tau]\to\mathcal M$, with $\eta(\tau)=x$. For sufficiently small $\varepsilon>0$, we may modify $\eta$ to ensure that
$$ \eta(t) = \Phi_{\tau,t}(x), \quad \forall t \in [\tau-\varepsilon,\tau].$$
By the compactness of $\gamma([0,1])$ and openness of $\mathcal{D}$, after possibly decreasing $\varepsilon$, we may also assume that every worldline through $\gamma([0,1])$ exists on $[\tau-\varepsilon,\tau]$. In particular, we may reparameterise $\gamma$ so that
$$ \gamma:[\tau-\varepsilon,\tau] \to \mathcal{M}_\tau, \quad \gamma(\tau-\varepsilon)=x, \quad \gamma(\tau) = y,$$
and $\gamma$ is constant near its endpoints. This reparameterisation yields the time-preserving path 
$$\zeta : [\tau-\varepsilon,\tau] \to \mathcal{M}, \quad \zeta(t) := \Phi_{\tau,t}\left( \gamma(t) \right).$$
The concatenation of $\eta\vert_{[s,\tau-\varepsilon]}$ and $\zeta$ is therefore a time-preserving path with endpoint $y$. Repeating for all such $\eta$, we have shown that $\mathfrak{h}(y) \leq \mathfrak{h}(x)$. Swapping the roles of $x$ and $y$ gives the reverse inequality.

\begin{cor}\label{Chen 2.3}
For $n \geq 2$, let $(\mathcal{M}^{n+1} , \mathfrak{t} , \partial_t , g)$ be a complete Ricci flow spacetime over $I$. Then 
\begin{equation*}
    \mathcal{R}_{g}(x) \geq \frac{-n}{2(\mathfrak{t}(x) - \mathfrak{h}(x))}, \quad \forall x \in \mathcal{M}.
\end{equation*}
(Here, we use the convention $\frac{1}{\infty} = 0$.)
\end{cor}
\begin{proof}
Fix $x \in \mathcal{M}$ and choose $s \in (\mathfrak{h}(x), \mathfrak{t}(x))$. By the definition of $\mathfrak{h}$ we find a time-preserving path $\eta:[s,\mathfrak{t}(x)] \rightarrow \mathcal{M}$ with $\eta(\mathfrak{t} (x)) = x$. For $r_0$ sufficiently small, $\Ric(g(t)) \leq (n-1)r_0^{-2}$ on $B_{g(t)}(\eta(t),r_0)$ for every $t \in [s, \mathfrak{t}(x)]$, and $|\pi (\eta'(t))|_{g(t)} \leq r_0^{-1}$ for every $t \in [s,\mathfrak{t}(x)]$. For any $A > 0$, and any $t \in [s,\mathfrak{t}(x)]$, $B_{g(t)}(\eta(t) , 2Ar_0) \Subset \mathcal{M}_t$ follows from the completeness of the metric $g(t)$. By compactness, there exists some lower bound $\mathcal{R}_{g(s)} \geq -K$ on $B_{g(s)}(\eta(s),2Ar_0)$. So, for each $\delta \in (0,\frac{2}{n})$, choosing $A$ sufficiently large, Theorem~\ref{Chen 2.1} implies the existence of $C>0$ such that
\begin{equation*}
\mathcal{R}_{g}(x) \geq -\max \{ \frac{1}{(\frac{2}{n}-\delta)(\mathfrak{t}(x)-s) + \frac{1}{K}}, \frac{C}{A^2r_0^2} \} \geq \frac{-1}{(\frac{2}{n}-\delta)(\mathfrak{t}(x)-s)}.
\end{equation*}
The conclusion follows after taking $\delta \downarrow 0$ and then $s \downarrow \mathfrak{h}(x)$.
\end{proof}

\begin{lem}
Let $(\mathcal M^{2+1},\mathfrak t,\partial_t,g)$ be a complete
Ricci flow spacetime over $I$, and fix $s,t \in I$ with $s<t$. Then we have the pointwise Harnack estimate on $\mathcal M_{s,t}$
\begin{equation}\label{weak harnack}
    \left(\Phi_{s,t}^{*}g(t)\right)_x \leq \frac{t-\mathfrak h(x)}{s-\mathfrak h(x)}\,g(s)_x, \quad \forall x \in \mathcal{M}_{s,t},
\end{equation}
with the quotient interpreted as $1$ when $\mathfrak h(x)=-\infty$. Consequently, if $\gamma:[0,1]\to\mathcal M_{s,t}$ is a smooth curve, then $\mathfrak h$ takes a constant value $h<s$ along $\gamma$, and
$$ L_{g(t)}(\Phi_{s,t}\circ\gamma)
\le
\sqrt{\frac{t-h}{s-h}}\,
L_{g(s)}(\gamma).
$$
\end{lem}

\begin{proof}
Fix $x\in\mathcal M_{s,t}$. For each $j \in \mathbb{N}$, set $h_j(x) := \max\{-j,\mathfrak{h}(x)\}$. Then for $j$ sufficiently large, $h_j(x)<s$, and since $\mathfrak{h}$ is non-increasing along worldlines, for every $r \in [s,t]$ we have
$$\mathfrak h(\Phi_{s,r}(x))\leq \mathfrak h(x)\leq h_j(x).$$
By Corollary~\ref{Chen 2.3}
$$ \mathcal R_{g(r)}(\Phi_{s,r}(x))
\ge -\frac{1}{r-\mathfrak h(\Phi_{s,r}(x))}
\ge -\frac{1}{r-h_j(x)}.$$
Therefore, the equation $\mathcal L_{\partial_t}g=-\mathcal R_g\,g$ yields
\begin{align*}
\frac{d}{dr} \big((r-h_j(x))^{-1} \cdot \left( \Phi_{s,r}^{*}g(r) \right)_x \big) &= (r-h_j(x))^{-1} \cdot \left( \Phi_{s,r}^{*} \mathcal{L}_{\partial_r} g(r) \right)_x -(r-h_j(x))^{-2} \cdot \left( \Phi_{s,r}^{*}g(r) \right)_x\\
&= \left[ - \mathcal R_{g(r)}(\Phi_{s,r}(x)) \cdot (r-h_j(x))^{-1} -(r-h_j(x))^{-2} \right] \left( \Phi_{s,r}^{*} g(r) \right)_x \leq 0.
\end{align*}
Integrating from $s$ to $t$
$$(\Phi_{s,t}^*g(t))_x \leq \frac{t-h_j(x)}{s-h_j(x)}\,g(s)_x.$$
Now let $j\to\infty$. If $\mathfrak h(x)>-\infty$, then eventually $h_j(x)=\mathfrak h(x)$,
whereas if $\mathfrak h(x)=-\infty$, then 
$$ \lim_{j \to \infty} \frac{t-h_j(x)}{s-h_j(x)} = \lim_{j \to \infty}  \frac{t+j}{s+j} =1.$$  This proves \eqref{weak harnack}. Finally, $\mathfrak{h} \vert_{\mathcal{M}_s}$ being locally constant means it takes a constant value $h<s$ along $\gamma$. The length estimate then follows by integrating \eqref{weak harnack} along $\gamma$.
\end{proof}

A finite-length curve in one time-slice cannot be stretched to infinite length at a later time, preventing worldlines from becoming extinct.

\begin{thm}\label{thm:expanding}
Let  $(\mathcal{M}^{2+1} , \mathfrak{t} , \partial_t , g)$ be a complete Ricci flow spacetime over $I=(\alpha,\beta)$. Then $\mathfrak t_+\equiv \beta$.
\end{thm}
\begin{proof}
Fix $s<t$ in $I$ and suppose there exists a smooth curve $\gamma:[0,1]\to\mathcal M_s$ such that
$$\gamma([0,1))\subset\mathcal M_{s,t}.$$
For every $\tau \in [s,t]$, the Harnack estimate implies that $\Phi_{s,\tau}\circ\gamma|_{[0,1)}$ has finite length and is Cauchy as we approach $1$. Completeness of $(\mathcal M_\tau,g(\tau))$ therefore gives 
$$ y(\tau):= \lim_{u \uparrow 1}\Phi_{s,\tau}(\gamma(u)) \in\mathcal M_\tau.$$

We claim that $\tau \mapsto y(\tau)$ is an integral curve of $\partial_t$. To see this, fix $\tau_0 \in [s,t]$. Since the maximal flow domain $\mathcal D$ is open and $(y(\tau_0),0)\in\mathcal D$, the flow $\Phi_{\tau_0,\tau}$ is defined in a neighbourhood of $y(\tau_0)$, for $\tau$ sufficiently close to $\tau_0$. Hence, for $u$ sufficiently close to $1$
$$\Phi_{s,\tau}(\gamma(u))=\Phi_{\tau_0,\tau} \left(\Phi_{s,\tau_0}(\gamma(u))\right).$$
Passing to the limit as $u\uparrow1$ gives $y(\tau)=\Phi_{\tau_0,\tau}(y(\tau_0))$.
Thus $y$ is an integral curve of $\partial_t$ on $[s,t]$. Since $y(s)=\gamma(1)$, it follows that $\gamma(1)\in\mathcal M_{s,t}.$

Hence $\mathcal M_{s,t}$ is a union of connected components of $\mathcal M_s$. In particular, for any pair of points $x,y \in \mathcal{M}_{s}$ lying in the same connected component, we have
$$ x \in \mathcal{M}_{s,t} \iff y \in \mathcal{M}_{s,t}, \quad \forall t \in (s,\beta).$$
But $x \in \mathcal{M}_{s,t}$ is equivalent to saying $\mathfrak{t}_+(q(x)) > t$, and so $\mathfrak{t}_+(q(x)) = \mathfrak{t}_+(q(y))$. By Lemma~\ref{lem: worldline} part (ii), since $q_s : \mathcal{M}_s \to \mathcal{W}_s$ is a homeomorphism, $\mathfrak{t}_+$ is locally constant on $\mathcal{W}_s$. Also, as $\mathcal{W}_s$ form an open cover of $\mathcal{W}$, it follows that $\mathfrak{t}_+$ is locally constant on $\mathcal{W}$. As Lemma~\ref{lem: worldline} part (iii) tells us $\mathcal{W}$ is connected, we know that $\mathfrak{t}_+$ is constant on $\mathcal{W}$. Finally, as $\mathfrak{t}$ is surjective, we must have $\mathfrak{t}_+ \equiv \beta$.
\end{proof}

\begin{cor}\label{thm emb}
Let $(\mathcal{M}^{2+1} , \mathfrak{t} , \partial_t , g)$ be a complete Ricci flow spacetime over $I=(\alpha,\beta)$. Then the space of worldlines $\mathcal{W}$ is a smooth connected surface, and
\begin{equation*}
  \iota:= (q,\mathfrak{t}) : \mathcal{M} \rightarrow \mathcal{W} \times I,
\end{equation*}
is a diffeomorphism onto an open subset of $\mathcal{W} \times I$. Moreover
$$\mathfrak{t} = t \circ \iota, \quad \iota _* (\partial_t) = \frac{\partial}{\partial t},$$ 
for the standard projection $t : \mathcal{W} \times I \rightarrow I$. Under this identification, the induced metrics satisfy the usual Ricci flow equation.
\end{cor}

\begin{proof}
By Theorem~\ref{thm:expanding}, $\mathfrak t_+(w)=\beta$ for every $w\in\mathcal W$, and hence $$\mathcal W_t = \{w\in\mathcal W:\mathfrak t_-(w)<t\}.$$
Thus the sets $\mathcal{W}_t$ form a nested increasing open cover of $\mathcal W$, and it follows that any two worldlines eventually lie in some common $\mathcal W_t$. By Lemma~\ref{lem: worldline}, every $\mathcal W_t$ is Hausdorff, so these worldlines can be separated in $\mathcal W_t$. As $\mathcal W_t$ is open in $\mathcal W$, it follows that $\mathcal W$ is Hausdorff. Combined with Lemma~\ref{lem: worldline}, the result follows.
\end{proof}

For the remainder of the paper we set $M:=\mathcal W$ and identify $\mathcal M$ with its image under $\iota$. Thus every complete $(2+1)$-dimensional Ricci flow spacetime will henceforth be written canonically as an open subset
$$ \mathcal M\subseteq M\times I, \qquad \partial_t=\frac{\partial}{\partial t}.$$

For each $t\in I$, our spatial slices are now identified as subsets of $M$
$$M_t := \{x\in M:(x,t)\in\mathcal M\},$$
with the explicit nesting
$$ M_t = \{ x \in M : \mathfrak{t}_-(x) < t \}, \quad M_s \subseteq M_t \quad (s \leq t).$$
Furthermore, since the Ricci flow preserves the conformal structure in dimension two, $M$ carries a unique conformal structure whose restriction to each $M_t$ agrees with the conformal class of $g(t)$.

With this identification, we have reduced Theorem~\ref{thm:main} to the following statement.

\begin{restatable}{thm}{scattered}\label{thm scattered}
For every $s \in I$, the closed set $\{ x \in M : \mathfrak{t}_-(x) \geq s \}$ is scattered in $M$. Consequently, $\{ x \in M : \mathfrak{t}_-(x) > \alpha\}$ is at most countable.
\end{restatable}

\section{Inception Measures}\label{section FD}

In this section we investigate the spacetime at times when new worldlines appear. That is, we fix a time $s \in I$ and study the points which are present immediately after $s$, regardless of if they already belong to $M_s$.

\begin{defn}
For $s \in I$ define the future domain
\begin{equation*}
    M_s^+ := \textnormal{Int} \left( \bigcap_{t \in (s,\beta)} M_t\right) \subseteq M.
\end{equation*}
\end{defn}

One may characterise $M_s^+$ as the largest open subset of $M$ for which the Ricci flow over the time interval $(s,\beta)$ is defined on. Although the metric $g(s)$ need not extend to $M_s^+$ in a smooth sense, its volume measure does extend to a meaningful measure on $M_s^+$.

\begin{thm}\label{thm:inception}
For every $s \in I$ there exists a Radon measure $\mu_s^+$ on $M_s^+$ such that 
$$\mu_{g(t)} \vert_{M_s^+} \rightharpoonup \mu_s^+, \quad \textnormal{as } t \downarrow s.$$
\end{thm}
\begin{rem}
Since $M_s \subseteq M_s^+$, we must have $\mu_s^+ \vert_{M_s} = \mu_{g(s)}$.
\end{rem}

This theorem is similar to \cite[Theorem 1.7]{peachey2025all}. However, it is worth noting a key technical distinction here: in our spacetime setting the restriction of the Ricci flow to the future domain is not in general complete, and without assuming any additional structure, we do not have access to a comparison principle between Ricci flows. As such, we need an alternative method to show finite local mass down to time $s$.  

\begin{proof}
After translating time, we may assume $s=0$. We work locally on a relatively compact coordinate neighbourhood $U \Subset M_0^+$. Fix $\tau \in (0,\beta)$ and write $g(t) = u(z,t) |dz|^2$ on $U \times (0,\tau]$. By the Harnack estimate \eqref{weak harnack}
\begin{equation*}
u(z,t) \geq \frac{t \cdot u(z,\tau)}{\tau} , \quad \forall (z,t) \in U \times (0,\tau].
\end{equation*}
Since $U \Subset M_{\tau}$ and $u(\cdot,\tau)$ is smooth and positive on $U$, there exists $\varepsilon>0$ such that
\begin{equation*}
u(z,t) \geq \varepsilon t, \quad \forall (z,t) \in U \times (0,\tau].
\end{equation*}

\begin{claim}
For every $K \Subset U$, we have
$$ \limsup_{t \downarrow 0} \mu_{g(t)}(K) < \infty.$$
\end{claim}
\begin{proof}[Proof of Claim]
Choose $\psi \in C^\infty_c(U)$ with $\psi \equiv 1$ on $K$, set $\phi := \psi^4$, and consider
$$ F(t) := \int_U \phi \ d\mu_{g(t)}.$$
A direct computation yields
\begin{align*}
F'(t) &= \int_U \phi \ \Delta \log u = \int_U \Delta\phi \ \log u = \int_U \Delta \phi \ \log\left(\frac{u}{\varepsilon t}\right),
\end{align*}
where the final equality follows from $\int \Delta\phi = 0$. Since $\log x \leq \sqrt{x}$ for $x \geq 1$, we have the bound
\begin{align*}
|F'(t)| &\leq \int \vert \Delta \phi\vert \sqrt{\frac{u}{\varepsilon t}}  \leq  (\varepsilon t)^{-\frac{1}{2}} \cdot \left( \int \frac{(\Delta \phi)^2}{\phi} \right)^{\frac{1}{2}} \cdot\sqrt{F(t)}.
\end{align*}
Using the formula $\phi = \psi^4$, it follows that
$$ \frac{|\Delta \phi|^2}{\phi} \lesssim \psi^2 |\Delta \psi|^2 + \vert \nabla \psi\vert^4 + \psi \vert \Delta \psi \vert \vert \nabla \psi \vert^2,$$
is a bounded compactly supported function, and so
$$ |F'(t)| \leq C t^{-\frac{1}{2}} \cdot\sqrt{F(t)}. $$
Integrating from $t$ to $\tau$, we have the upper bound
$$ \mu_{g(t)}(K) \leq  F(t) \leq \left(\sqrt{F(\tau)} + C \sqrt{\tau } \right)^2< \infty, \quad \forall t \in (0,\tau].$$
\end{proof}
As a consequence of the Claim combined with weak compactness for measures, every sequence $t_n \downarrow 0$ has a subsequence with
$$ \mu_{g(t_n)} \rightharpoonup \mu_0^+, \quad \textnormal{as } n \to \infty,$$
for some Radon measure $\mu_0^+$ on $U$. To see that such a limit is unique, for any $\phi \in C^\infty_c(U)$, the same calculation gives
\begin{align*}
\left| \frac{d}{dt} \int \phi d\mu_{g(t)} \right| &= \left|  \int \Delta \phi \log \left( \frac{u}{\varepsilon t} \right) \right| \leq \frac{1}{\sqrt{\varepsilon t}}   \int \left|\Delta \phi \right| \sqrt{u}\\
&\leq \frac{1}{\sqrt{\varepsilon t}}   \left( \int (\Delta \phi)^2 \right)^{\frac12} \left( \int_{\supp \phi} u \right)^{\frac12} \leq C t^{-1/2},
\end{align*}
where we have used that the mass is bounded in the final inequality. Since this is integrable, it follows that
\begin{equation}\label{eq:weak conv}
\int \psi d\mu_{g(t)} \rightarrow \int \psi d\mu_0^+ ,\quad \text{as } t \downarrow 0,
\end{equation}
for all $\psi \in C_c^\infty(U)$. Now for any $\psi \in C^0_c(U)$, we can consider the mollification $\psi_h$, which has support uniformly in some $K \Subset U$ for sufficiently large $h$, and the estimate
\begin{equation*}
\left| \int \psi d\mu_{g(t)} - \int \psi d\mu_0^+ \right| \leq \left| \int \psi_h d\mu_{g(t)} - \int \psi_h d\mu_0^+ \right| + \norm{\psi - \psi_h}_{C^0} \cdot ( \mu_{g(t)}(K) + \mu_0^+(K)).
\end{equation*}
It now follows from the uniform mass bound and the uniform convergence of $\psi_h \to \psi$, that $\mu_{g(t)} \rightharpoonup \mu_0^+$ as $t \downarrow 0$ on $U$. Finally, the global limit arises from a standard partition of unity argument, thus completing the proof.
\end{proof}

\begin{lem}\label{lem:fubini}
Let $(\Omega,\bar g)$ be a smooth connected surface with
$E \subset \Omega$ an open non-empty subset. Suppose that $E$ admits a smooth complete metric $g$ conformal to $\bar g$, satisfying
$$\mu_g(E \cap K)<\infty \qquad\text{for every } K \Subset \Omega.$$
Then $\mu_{\bar g}(\Omega \setminus E)=0.$
\end{lem}

\begin{proof}
The argument is local. Fix a conformal coordinate chart $U \Subset \Omega$ with complex coordinate $z$, and write
$$g=v(z)|dz|^2 \quad \textnormal{on }  E \cap U.$$

Fix an open rectangle $R:= I \times J \Subset U$ with $R \cap E \neq \emptyset$. For any $x \in I$ and $y \in J$ define the subsets
$$ I_y := \{ x \in I : (x,y) \in E\}, \quad J_x := \{ y \in J : (x,y) \in E\}.$$
By Fubini, for almost every $(x,y) \in I \times J$ we have both
$$ \int_{I_y} v(x,y) dx < \infty , \quad \int_{J_x} v(x,y) dy < \infty.$$
Fix such a point $(x,y)$. Since
$$ \int_{I_y} \sqrt{v}(x,y) dx \leq \vert I_y \vert^{\frac{1}{2}} \cdot \left(\int_{I_y} v(x,y) dx\right)^{\frac{1}{2}} < \infty,$$
it follows that every connected component of $I_y$ can have no endpoints in the interior of $I$, as otherwise the corresponding horizontal segment would approach $\partial E$ with finite $g$-length. So either $I_y$ is empty or equal to $I$. Thus, there exists $a(y) \in \{0,1\}$ for almost every $y \in J$ such that
$$ \mathrm{1}_{E}(x,y) = a(y), \quad \forall x \in I.$$
By repeating the exact same argument, there exists $b(x) \in \{0,1\}$ for almost every $x \in I$ such that
$$ \mathrm{1}_{E}(x,y) = b(x), \quad \forall y \in J.$$
Combining these two, we deduce that $\mathrm{1}_{E}$ is constant almost everywhere on the rectangle $R$. 
Since the intersection of $E$ and $R$ is open and non-empty, this indicator must be equal to $1$ almost everywhere on $R$. In other words, $\mu_{\bar{g}}(R \setminus E) = 0$. 

Next, consider all such rectangles
$$ \mathcal{R} :=\bigcup\left\{
R\Subset \Omega:\ R\text{ is an open rectangle with }R\cap E\neq\emptyset
\right\}.$$
Immediately we have $\mathcal{R}$ is non-empty and open. In fact it is also closed in $\Omega$. Indeed, suppose $p_k\in \mathcal{R}$ and $p_k\to p\in \Omega$. Choose an open rectangle
$R\Subset \Omega$ containing $p$. For $k$ sufficiently large, $p_k\in R$.
Since $p_k\in \mathcal{R}$, there is an open rectangle $R_k\Subset \Omega$ containing $p_k$ with $R_k\cap E \neq\emptyset$. By the argument above
$$\mu_{\bar{g}}(R_k\setminus E)=0.$$
Since $R \cap R_k$ is a non-empty open set, it has positive measure, and hence cannot be contained in the null set $R_k\setminus E$. Therefore $R\cap E\neq\emptyset$ and $p\in \mathcal{R}$.

Since $\Omega$ is connected, $\mathcal{R} = \Omega$. Finally, we may now take a countable cover of $\Omega$ by rectangles $R_j$ all of which have non-empty intersection with $E$, which shows
$$ \mu_{\bar{g}}(\Omega \setminus E) \leq \sum_{j=1}^\infty \mu_{\bar{g}}(R_j \setminus E) = 0.$$
\end{proof}

\begin{cor}\label{cor:different cc}
For each $s \in I$, a connected component of $M_s^+$ contains at most one connected component of $M_s$.
\end{cor}

\begin{proof}
Suppose $E$ is a connected component of $M_s$ sitting inside a connected component $\Omega$ of $M_s^+$. Note that $g(s)$ is a complete metric on $E$ and by Theorem~\ref{thm:inception} we have
\begin{equation*}
\mu_{g(s)}(E \cap K) = \mu_s^+(E \cap K) \leq \mu_s^+(K) < \infty,
\end{equation*}
for every $K \Subset \Omega$. Therefore the hypotheses of Lemma~\ref{lem:fubini} are satisfied and we may conclude 
$\mu_{\bar{g}}(\Omega \setminus E) = 0$. In particular, $\Omega$ cannot contain a different connected component of $M_s$.
\end{proof}

If two points suddenly become connected for the first time at $\sigma \in I$, then immediately afterwards, the spatial slices are hyperbolic domains, and the hyperbolic distance between the two points diverges as $t\downarrow\sigma$. The Schwarz lemma transfers this divergence to the spacetime metric, contradicting Perelman's distance distortion estimate.

\begin{lem}\label{lem uni planar}
Let \(U\subseteq V\) be connected conformal surfaces, with \(U\) open in \(V\). If \(U\) is non-hyperbolic, then \(V\setminus U\) contains at most two points.
\end{lem}

\begin{proof}
We first assume $V$ is orientable, and hence a Riemann surface. Since every open subset of a hyperbolic Riemann surface is hyperbolic, the assumption that $U$ is non-hyperbolic implies that $V$ is non-hyperbolic. By the uniformisation theorem, $V$ is therefore conformally equivalent to either the sphere, the plane, the punctured plane, or a complex torus.  In all four cases, removing three points produces a hyperbolic surface. This proves the statement in the case $V$ is orientable. In the case $V$ is non-orientable, simply pass to the double cover and apply the previous case to reach the same conclusion.
\end{proof}

\begin{lem}[Schwarz Lemma \cite{yau1973remarks}]\label{Schwarz-Pick}
Let $(M_1,g_1)$, $(M_2,g_2)$ be Riemannian surfaces. If
\begin{enumerate}
    \item $(M_1,g_1)$ is complete,
    \item the Gaussian curvature $K(g_1) \geq -a_1$, for some $a_1 \geq 0$,
    \item $K(g_2) \leq -a_2$, for some $a_2 > 0$,
\end{enumerate}
then any conformal map $f:M_1 \rightarrow M_2$ satisfies the inequality.
\begin{equation*}
   f^*(g_2) \leq \frac{a_1}{a_2} g_1. 
\end{equation*}
\end{lem}

\begin{lem}\label{lem hyp sep}
Let $M$ be a connected conformal surface and $\Omega_j \subseteq M$ a sequence of decreasing nested domains. Set
\[
K:=\left(\bigcap_j\Omega_j\right)^\circ.
\]
If $K$ is disconnected, then $\Omega_j$ is eventually hyperbolic. Moreover, if $x,y\in K$ lie in different connected components of $K$, then $d_{h_j}(x,y)\to\infty$ as $j \to \infty$, where for sufficiently large $j$, $h_j$ denotes the complete hyperbolic metric on $\Omega_j$.
\end{lem}

\begin{proof}
The first statement follows from Lemma~\ref{lem uni planar}. Indeed, if the $\Omega_j$ are not eventually hyperbolic, since an open subset of a hyperbolic surface is hyperbolic, none of them are hyperbolic. Thus $M\setminus\Omega_j$ contains at most two points for every $j$. These sets are nested and increasing, and so eventually stabilise. Thus the $\Omega_j$ eventually stabilise, which contradicts the assumption that $K$ is disconnected.

Next, let $U$ be the connected component of $K$ containing $x$. Choose $j_0$ so that $\Omega_j$ is hyperbolic for every $j\geq j_0$, and fix a universal covering
$$\pi:D\longrightarrow\Omega_{j_0}, \quad \pi(0)=x.$$
For $j\geq j_0$, let $\widetilde\Omega_j$ be the connected component of $\pi^{-1}(\Omega_j)$ containing $0$. These form a decreasing nested sequence of domains in $D$. If $\widetilde U$ denotes the connected component of $\left(\bigcap_{j\geq j_0}\widetilde\Omega_j\right)^\circ$ containing $0$, then $\pi(\widetilde U)\subset U$. For all $j \geq j_0$, let $f_j:D \to \widetilde \Omega_j$ be the universal covering, normalised by
$$f_j(0)=0, \quad f_j'(0) \in (0,\infty).$$
By the Carathéodory kernel theorem for universal covering maps (\cite[Theorem 1]{Hejhal1974}, see also \cite[Theorem 7]{hoehn2019extension}), the $f_j$ converge locally uniformly to the universal covering
$f: D \to \widetilde U$. If $d_{h_j}(x,y)$ were bounded along a subsequence, we could choose
$z_j\in D$ with $\pi(f_j(z_j))=y$ and $d_h(0,z_j) = d_{h_j}(x,y)$ uniformly bounded, where $h$ denotes the standard Poincar\'e metric on the disk. Since the points $z_j$ must lie in a compact subset of the unit disk, passing to a further subsequence, we may extract a limit $z_j\to z\in D$, and hence we obtain the contradiction
$$y=\lim_{j \to \infty} \pi(f_j(z_j)) =\pi(f(z)) \in \pi(\widetilde U) \subset U.$$
\end{proof}

\begin{thm}\label{thm: spatial connected}
$M_s$ is connected for every $s \in I$.
\end{thm}
\begin{proof}
Fix $x,y \in M_s$ and consider the collection of times
$$ J := \{ t \in [s,\beta) : x,y \textnormal{ are path-connected in } M_t \}.$$
Since the spatial slices are nested, $t \in J$ implies $[t,\beta) \subseteq J$. Moreover, for any rectifiable curve connecting $x$ and $y$ in $M$, since its image is compact, we know that the curve is eventually contained inside $M_t$ for $t$ sufficiently large. Therefore $J$ is non-empty and there is a well-defined infimum $\sigma := \inf J$. 

We claim that $x$ and $y$ are path-connected in $M_{\sigma}$. If not, they lie in different connected components of $M_{\sigma}$. By Corollary~\ref{cor:different cc} they must also lie in different connected components of $M_{\sigma}^+$. Lemma~\ref{lem hyp sep} tells us that $M_t$ is hyperbolic for $t$ sufficiently close to $\sigma$, and
$$ d_{h_t}(x,y) \uparrow \infty,$$
as $t \downarrow \sigma$, where $h_t$ denotes the complete hyperbolic metric on $M_t$. Let $s_0 \in (\alpha,s)$ be such that $\mathfrak{t}_-(x) < s_0$. In particular, for $t > \sigma$, let $C_t$ be the connected component of $M_t$ containing $x$ and $y$. Since $\mathfrak h$ is constant on $C_t$ and $\mathfrak h((x,t)) < s_0$, we have $\mathfrak h\vert_{C_t} < s_0$, and therefore
$$ \mathcal{R}_{g(t)} \geq -\frac{1}{t-s_0}, \quad \textnormal{on } C_t.$$
Applying the Schwarz lemma gives $g(t) \geq 2(t-s_0) h_t$ on $C_t$, and hence 
\begin{equation}\label{eq:div}
d_{g(t)}(x,y) \uparrow \infty, \quad \textnormal{as } t \downarrow \sigma.  
\end{equation}
On the other hand, as $g$ is smooth there exists $r,R > 0$ such that, for every $t \in (\sigma,\sigma+r^2)$, we have the bound $\mathcal{R}_g \leq R$ on the disjoint metric balls
$$B_{g(t)}(x,r) \bigcup B_{g(t)}(y,r) \subseteq \mathcal{M}.$$
Thus, applying  \cite[][Lemma 8.3]{perelman2002entropy} we have the lower bound
$$ \frac{d}{dt} d_{g(t)} (x,y) \geq - C(r,R),$$
for every $t \in (\sigma, \sigma+r^2)$, which is a contradiction to the divergence in \eqref{eq:div}.

It now follows that $\sigma = s$. If not and $\sigma > s$, then by compactness of the path joining them in $M_{\sigma}$, this same path would lie in $M_t$ for $t$ sufficiently close to $\sigma$, directly contradicting the definition of $\sigma$. This completes the proof of Theorem~\ref{thm: spatial connected}.
\end{proof}

An immediate application of the spatial slices being connected is that the hindsight function must equal the initial time $\alpha$ everywhere.

\begin{cor}\label{cor:hindsight}
Let $(\mathcal M^{2+1},\mathfrak{t},\partial_t,g)$ be a complete Ricci flow spacetime over $I$. Then its hindsight function satisfies $\mathfrak h\equiv\alpha$. In particular, the lower scalar curvature bound from Corollary~\ref{Chen 2.3} becomes
$$ \mathcal{R}_g(x) \geq \frac{-1}{\mathfrak{t}(x) - \alpha}, \quad \forall x \in \mathcal{M}.$$
\end{cor}

\begin{proof}
Recall that $\mathfrak h$ is locally constant on each spatial
slice. Fix $x \in \mathcal{M}$. For any $y \in \mathcal{M}$ with $\mathfrak{t}(y) \leq \mathfrak{t}(x)$, since $\mathfrak{t}_+ \equiv \beta$, the worldline $q(y)$ through $y$ exists at time $\mathfrak{t}(x)$. Let $y' \in \mathcal{M}_{\mathfrak{t}(x)}$ be the point in this spatial slice on this worldline, i.e. $q(y') = q(y)$. Now, as every spatial slice is connected and the hindsight function is non-increasing along worldlines
$$ \mathfrak{h}(x) = \mathfrak{h}(y') \leq  \mathfrak{h}(y) \leq \mathfrak{t}(y).$$
Taking the infimum over all such $y$ gives $\mathfrak{h}(x) = \alpha$.
\end{proof}

Under non-negative curvature, no new worldlines may appear.

\nn*

\begin{proof}
As a consequence of the Schwarz Lemma, every spatial slice must be non-hyperbolic. In particular, applying Lemma~\ref{lem uni planar} we find that the number of points with inception time larger than $\alpha$ is at most two, and so each such puncture is isolated.

Suppose $p \in M$ with $\mathfrak{t}_-(p) = s > \alpha$. Then we may choose a small neighbourhood $p \in D \Subset M$, admitting a complex coordinate $z$ with $z(p) = 0$, such that $D \cap M_s = D \setminus \{p\}$, and $D \cap M_t = D$, for all $t \in (s,\beta)$. Since $g(s)$ is complete and has non-negative curvature, by the Cohn-Vossen inequality it has finite total curvature. We may now appeal to a result of Huber \cite[Satz~1]{Huber1966} (see also \cite[Propositions~1.5 and~1.7]{HulinTroyanov1992}) to deduce that, in this local complex coordinate, the complete non-negatively curved metric $g(s)$ has the form 
$$g(s) = e^{2u(z)} \abs{z}^{-2\sigma} \cdot \abs{dz}^2, \quad \forall z \in D^\times,$$
where $\sigma \geq 1$ and $u \in L^1(D)$ is superharmonic. By the lower semicontinuity of $u$, it is uniformly bounded below on $D_{\frac{1}{2}} \Subset D$, and hence $e^{-2u} \in L^\infty(D_{\frac{1}{2}})$. In particular
\begin{align*}
\infty = \int_{D_{\frac{1}{2}}} \abs{z}^{-2\sigma} \abs{dz}^2 = \int_{D_{\frac{1}{2}}} e^{-2u(z)} e^{2u(z)} \abs{z}^{-2\sigma} \abs{dz}^2 \leq C \int_{D_{\frac{1}{2}}} e^{2u(z)} \abs{z}^{-2\sigma} \abs{dz}^2,
\end{align*}
for some finite constant $C < \infty$. We have shown that $\mu_{g(s)}(D^\times) = \infty$. However, by Theorem~\ref{thm:inception}
$$ \mu_{g(s)}(D^\times) = \mu_s^+(D^\times) \leq \mu_s^+(D) < \infty,$$
which is a contradiction. Thus $\mathfrak{t}_- \equiv \alpha$, and the isomorphism with a complete Ricci flow on a fixed smooth surface $M$ is given precisely by the map $\iota$ in Corollary~\ref{thm emb}.
\end{proof}

\section{Future Domains}
Fix a smooth complete conformal metric $\bar g$ on $M$.
\begin{defn}
For $s \in I$, define
$$ \Theta_s^+(p) := \limsup_{r \downarrow 0} \frac{\mu_s^+(B_{\bar{g}}(p,r))}{\pi r^2}, \quad \forall p \in M_s^+,$$
\end{defn}
At almost every point of $M_s^+$, the inception measure has finite two-dimensional density: working in local coordinates, by the Lebesgue decomposition theorem \cite[Theorem 1.31]{evansgariepy2015}
\begin{equation*}
\mu_s^+ = f_s \cdot d\mu_{\bar{g}} + \nu_s^+,
\end{equation*}
where $f_s \in L^1_{loc}(M_s^+)$ and $\nu_s^+ \perp \mu_{\bar{g}}$.
By the differentiation theorem for Radon measures \cite[Theorem 1.32]{evansgariepy2015}
$$ \lim_{r \downarrow 0} \frac{\mu_s^+(B_{\bar g}(p,r))}
{\pi r^2} =\lim_{r \downarrow 0} \frac{\mu_s^+(B_{\bar g}(p,r))}
{\mu_{\bar g}(B_{\bar g}(p,r))}
= f_s(p),$$
for $\mu_{\bar g}$-almost every $p$. In particular, $\Theta_s^+$ is finite for $\mu_{\bar{g}}$-almost every $p \in M_s^+$.

\begin{defn}
Given $(p,s) \in M \times I$ and positive numbers $r,\alpha>0$, define the parabolic cone centred at $(p,s)$ of radius $r$ and aperture $\alpha$ by
\begin{equation*}
C^+(p,s;\alpha,r) := \{ (x,t) \in M \times I : t \in (s,s+r^2), \ d_{\bar{g}}(p,x) < \alpha \sqrt{t-s}\}.
\end{equation*}
\end{defn}
\begin{rem}
For any $p \in M_s^+$, by shrinking $r$ we may ensure that $B_{\bar{g}}(p,\alpha r) \Subset M_s^+$, and $B_{\bar g}(p,\alpha r)\times(s,s+r^2) \subset \mathcal M$. In particular, $C^+(p,s;\alpha,r) \subset \mathcal{M}$ for $r$ sufficiently small.
\end{rem}

The following is the $L^1 - L^\infty$ smoothing result \cite[Theorem 2.1]{topping2024smoothing} generalised to the setting of weak initial data.
\begin{thm}\label{thm l1-linfty}
Suppose $g(t) = u(z,t) |dz|^2$ is a smooth conformal Ricci flow on the ball $B_{3r} \times (0,T)$, for some $r>0$. Suppose there exists a Radon measure $\mu$ on $B_{3r}$ such that
$$d\mu_{g(t)}\rightharpoonup\mu
\quad\text{as } t \downarrow 0.$$
If $t \in (0,T)$ satisfies $t > \frac{\mu(\overline{B_{2r}})}{2\pi}$ then
\begin{equation*}
\sup_{z \in B_r} u(z,t) \leq C_0 r^{-2} t, 
\end{equation*}
where $C_0 < \infty$ is universal.
\end{thm}

\begin{proof}
Since $\mu$ is Radon
$$t > \frac{\mu(B_{2r+\delta})}{2\pi}, \quad \textnormal{for some } \delta>0.$$
Choose a cut off $f$ with $f \equiv 1$ on $\overline{B_{2r}}$ and $\supp(f) \subseteq B_{2r+\delta}$. For $\varepsilon > 0$ sufficiently small
\begin{equation*}
\mu_{g(\varepsilon)}(B_{2r}) \leq \int f d\mu_{g(\varepsilon)} \leq \int f d\mu + \left( 2\pi t - \mu(B_{2r+\delta}) \right) \leq 2\pi t.
\end{equation*}
Therefore applying the estimate in the smooth case \cite[][Theorem 2.1]{topping2024smoothing} to the Ricci flow on $B_{2r} \times (\varepsilon, T)$ gives
\begin{equation*}
     \sup_{z \in B_r} u(z,\varepsilon + t) \leq C_0 r^{-2} t.
\end{equation*}
Since $u$ is smooth away from zero, taking $\varepsilon \downarrow 0$ gives the result.
\end{proof}

\begin{lem}\label{lem:p inc}
Fix $s\in I$ and $p\in M_s^+$. Suppose $\Theta_s^+(p)<\infty$. Then there exist $\alpha,r,C>0$ such that $C^+(p,s;\alpha,r) \subset \mathcal{M}$, and writing $g(t)=v(\cdot,t)\bar g,$ one has
\begin{equation}
v(x,t)\le C, \quad  \forall (x,t)\in C^+(p,s;\alpha,r).
\end{equation}
Furthermore, if we have a uniform positive lower bound $v\ge c>0$ on $C^+(p,s;\alpha,r)$, then there exists $C_1<\infty$ such that
$$|\mathcal R_g(x,t)|
\leq \frac{C_1}{t-s}, \quad \forall (x,t)\in C^+(p,s;\alpha/4,r).$$

\end{lem}

\begin{proof}
Choose a relatively compact conformal coordinate neighbourhood $U\Subset M_s^+$ of $p$, with $z(p)=0$. Since smooth metrics
are uniformly equivalent on $U$, finiteness of $\Theta_s^+(p)$ and the conclusion of the lemma are unchanged, up to modifying the
constants $\alpha,r$ and $C$, if we replace $\bar g$ on $U$ by the Euclidean metric. Therefore, without loss of generality we may assume that on $U$
$$\bar g=|dz|^2 \quad \textnormal{and} \quad g(t)=v(\cdot,t)|dz|^2.$$

The assumption $\Theta_s^+(p)<\infty$ gives constants $A,\rho_0>0$ such that
$$\mu_s^+(\overline{B_\rho}) \leq A\rho^2,
\quad \textnormal{for every } \rho \in (0,\rho_0),$$
where $B_\rho$ denotes the Euclidean ball centred at $0$. Choose $\alpha^2 \in (0, \frac{\pi}{2A})$. After decreasing $r>0$, we may assume that
$B_{3\alpha r}\Subset U\cap M_s^+$ and $3\alpha r<\rho_0$. Fix $\tau \in (0,r^2)$ and put $\rho:=\alpha\sqrt{\tau}$. The time-shifted flow $\tilde g(t):=g(s+t)$ is a smooth conformal Ricci flow on $B_{3\rho}\times(0,r^2)$ starting weakly from $\mu_s^+|_{B_{3\rho}}$, with
\begin{equation}
  \frac{\mu_s^+(\overline{B_{2\rho}})}{2\pi}
\le
\frac{2A\rho^2}{\pi}
=
\frac{2A\alpha^2}{\pi}\tau
<\tau.  
\end{equation}
In particular, Theorem~\ref{thm l1-linfty} yields
$$ \sup_{z \in B_\rho} v(z,s+\tau) \leq C_0\rho^{-2}\tau = C_0\alpha^{-2}.$$
Since $\tau\in(0,r^2)$ was arbitrary, $v$ is uniformly bounded on $C^+(p,s;\alpha,r)$. 

Suppose now that $v\ge c>0$ on $C^+(p,s;\alpha,r)$. For each $\tau \in (0,r^2)$ we parabolically rescale about $(p,s)$ by setting
$$ v_\tau(y,u)
:=
v(\sqrt{\tau}\,y,s+\tau u),$$
so that we have uniform upper and lower bounds on a fixed cylinder
$$ c\le v_\tau(y,u)\le C, \quad \forall (y,u) \in B_{\alpha/2} \times [1/2,1].$$
The logarithm of the conformal factor $w_\tau:=\log v_\tau$ solves the parabolic equation
\begin{equation*}
\partial_u w_\tau = e^{-w_\tau}\Delta w_\tau,
\end{equation*}
on $B_{\alpha/2} \times [1/2,1]$, with $e^{-w_\tau}$ bounded uniformly. Standard interior parabolic regularity therefore gives
$$ |\partial_u w_\tau| \leq C_1, \quad \textnormal{on } B_{\alpha/4}\times[3/4,1],$$
where $C_1$ is independent of $\tau$. Since $\mathcal R_g=-\partial_t\log v$, evaluating at $u=1$ yields
$$ |\mathcal R_g(\sqrt{\tau}\,y,s+\tau)| = \tau^{-1}|\partial_u w_\tau(y,1)| \leq \frac{C_1}{\tau},$$
for $|y|<\alpha/4$. As $\tau=t-s$, this proves
$$|\mathcal R_g| \leq \frac{C_1}{t-s} \quad \textnormal{on} \quad C^+(p,s;\alpha/4,r).$$
\end{proof}

A priori, $M_s^+$ could contain a component disjoint from $M_s$, which exists for every $t>s$ but disappears entirely at $t=s$. 

\begin{lem}\label{lem: fd intersects ss}
For every $s \in I$, each connected component of $M_s^+$ has non-empty intersection with $M_s$.
\end{lem}

\begin{proof}
After translating time, we may assume $0 \in (\alpha,s)$. If the lemma is false, then there is a connected component $\mathcal C$ of $M_s^+$ with $\mathcal C \cap M_s = \emptyset$. Since $M_s^+$ is  disconnected, by Lemma~\ref{lem hyp sep}, $M_t$ is hyperbolic for $t>s$ sufficiently close to $s$. If $h_t$ denotes the complete hyperbolic metric on $M_t$ for such values of $t$, then by the Schwarz lemma $g(t) \geq 2t h_t$, and hence by the nested property of the $M_t$ we have the lower bound 
\begin{equation}\label{eq:lbfc}
g(t) \geq 2s \cdot h_{t_0}, \quad \forall t \in (s,t_0].
\end{equation}
Choose $x \in C$ with $\Theta_s^+(x) < \infty$ and $y \in M_s$. As $\mu_s^+ = \mu_{g(s)}$ on $M_s$, we also have $\Theta_s^+(y) < \infty$. Using Lemma~\ref{lem hyp sep} again, together with \eqref{eq:lbfc}, we have the divergence
\begin{equation}\label{eq:mdiv1}
    d_{g(t)}(x,y) \geq \sqrt{2s} \cdot d_{h_t}(x,y) \uparrow \infty, \quad \textnormal{as } t \downarrow s.
\end{equation}
Since $h_{t_0}$ and $\bar{g}$ are uniformly equivalent on relatively compact neighbourhoods of $x$ and $y$ in $M_s^+$, by \eqref{eq:lbfc} it follows that
$$g(t) \geq c_0^2 \cdot \bar{g}, \quad \forall t \in (s,t_0],$$
on such neighbourhoods, for some $c_0> 0$. By Lemma~\ref{lem:p inc}, we may then deduce the existence of $\alpha,r,C > 0$, with $s+r^2< t_0$ and
\begin{equation*}
   \vert \mathcal{R}_g \vert \leq \frac{C}{t-s} \quad \textnormal{on} \quad  C^+(x,s;\alpha,r) \cup C^+(y,s;\alpha,r).
\end{equation*}
Shrinking $r$ if necessary, we may take $C^+(x,s;\alpha,r) \cap C^+(y,s;\alpha,r) = \emptyset$. Using the constant $c_0$, we also have the inclusion of the metric balls
$$ B_{g(t)}(x,c_0 \alpha \sqrt{t-s}) \subset C^+(x,s;\alpha,r), \quad B_{g(t)}(y,c_0 \alpha \sqrt{t-s}) \subset C^+(y,s;\alpha,r),$$
for all $t \in (s,s+r^2)$. Thus, we may now apply \cite[][Lemma 8.3]{perelman2002entropy} to give the lower bound
\begin{equation*}
\frac{d}{dt} d_{g(t)}(x,y) \geq - C(t-s)^{-\frac{1}{2}}, \quad \forall  t \in (s, s+r^2).
\end{equation*} 
Integrating yields
$$ d_{g(t)}(x,y) \leq d_{g(s+r^2/2)}(x,y) + 2Cr < \infty,$$
for all $t \in (s,s+r^2/2)$, which is a contradiction to the divergence in \eqref{eq:mdiv1}.
\end{proof}

\begin{thm}\label{thm:fld main}
For every $s \in I$, both $M_s$ and $M_s^+$ are connected. Furthermore, for any complete conformal metric $\bar{g}$ on $M$, we have 
$$ \mu_{\bar{g}}(M_s^+ \setminus M_s) = 0.$$
\end{thm}

\begin{proof}
By Theorem~\ref{thm: spatial connected} $M_s$ is connected. By Lemma~\ref{lem: fd intersects ss} $M_s$ intersects every connected component of $M_s^+$. It follows immediately that $M_s^+$ is also connected. Moreover, the hypotheses of Lemma~\ref{lem:fubini} now apply to $M_s$ inside $M_s^+$ and hence $\mu_{\bar{g}}(M_s^+ \setminus M_s) = 0$.
\end{proof}

\section{Comparison of Ricci flows}

We now prove a comparison principle needed for a barrier argument. Unlike for a fixed domain, a comparison region may escape through the moving boundary. For a spacetime within the unit disk, the key point is to show that the conformal factor diverges near every such escape.

Throughout this section we work on the universal cover of a hyperbolic time-slice. In particular, we assume $M$ is the unit disk equipped with its standard conformal structure. 

\begin{defn}\label{defn pp}
Let $(\mathcal{M},\mathfrak{t},\partial_t,g)$ be a Ricci flow spacetime over $I$. A function $f : \mathcal{M} \rightarrow \mathbb{R}$ is  \textit{parabolically proper} if $(f,\mathfrak{t}): \mathcal{M} \to \mathbb{R} \times I$ is proper.
\end{defn}

\begin{thm}\label{thm:cf pp}
Suppose $\mathcal{M} \subset D \times (0,T)$ is a complete Ricci flow spacetime, and write
$$ g(t) = v(z,t) |dz|^2 ,\quad \forall (z,t) \in \mathcal{M}.$$
Then $v$ is parabolically proper on $\mathcal{M}$.
\end{thm}

We repeatedly use the following simple consequence of the Schwarz Lemma.
\begin{lem}\label{lem uhlb}
For any $\Omega \subseteq D$ open, $x \in \Omega$ and $y \in \overline{D} \setminus \Omega$, we have the lower bound
$$ H_{\Omega}(x) \geq \frac{1}{|x-y|^2 \left(\log\frac{2}{|x-y|}\right)^2},$$
where $h_{\Omega} = H_{\Omega}(z) |dz|^2$ denotes the complete hyperbolic metric on $\Omega$.
\end{lem}
\begin{proof}
Consider the complete hyperbolic metric on the punctured disk $B(y,2) \setminus \{y\}$. By the Schwarz lemma, this sits below $h_{\Omega}$. The inequality follows from simply evaluating their conformal factors at $x$.    
\end{proof}

\begin{lem}\label{lem: VD}
Suppose $(z_i,t_i)\in \mathcal M$ satisfies
\[
t_i\to s\in (0,T),
\qquad
\sup_i v(z_i,t_i)<\infty.
\]
After passing to a subsequence, $z_i\to z_\infty\in D$. Moreover:
\begin{enumerate}[label=(\roman*)]
    \item if $t_i \leq s$ for infinitely many $i$, then
    $z_\infty \in M_s$;
    \item if $t_i>s$ for all sufficiently large $i$, then
    $z_\infty \in M_s^+$.
\end{enumerate}
Consequently, if $(z_i,t_i)$ leaves every compact subset of
$\mathcal M$, then, after passing to a subsequence,
$$t_i \downarrow s, \quad z_\infty\in M_s^+\setminus M_s.$$
\end{lem}
\begin{proof}
Choose $a > 0$ such that $t_i \geq a$ for every $i$. By the Schwarz Lemma, $g(t_i) \geq 2a \cdot h_D$ where $h_D$ is the Poincar\'e metric on the disk. By Lemma~\ref{lem uhlb}, after passing to a subsequence we find $z_i \to z_\infty \in D$. For (i), $h_{t_i} \geq h_s$ by the Schwarz Lemma, and again by Lemma~\ref{lem uhlb} we must have $z_\infty \in M_s$. For (ii), $g(t_i) \geq 2s \cdot h_{t_i}$ by the Schwarz Lemma, and yet again by Lemma~\ref{lem uhlb} $z_\infty \in M_s^+$. Indeed, if we could find $y_i \in D \setminus M_{t_i}$ with $|y_i-z_\infty| \to 0$, then $|y_i-z_i| \to 0$ and hence $v(z_i,t_i)$ diverges.
\end{proof}

Completeness alone only gives divergence along spatial curves; the following argument upgrades this to general spacetime curves near the future limit domain. In lieu of a comparison principle we must use direct PDE estimates.

\begin{lem}\label{lem:LLFDE}
Let $\Omega \subseteq \mathbb{C}$ and $u:\Omega \times (0,T) \to (0,\infty)$ be a smooth solution to the LFDE
\begin{equation*}
\partial_t u = \Delta \log u.
\end{equation*}
Suppose that $u\geq \delta>0$ on $\Omega\times(0,T)$, and that, for some $c>0$,
\[
    \liminf_{t\downarrow0}u(z,t)\geq c, \quad\text{for a.e. }z\in\Omega.
\]
Then, for every $K\Subset\Omega$,
\[
    \liminf_{t\downarrow0}\inf_{z\in K}u(z,t)\geq c.
\]
\end{lem}

\begin{proof}
If $\delta \geq c$, then there is nothing to prove, so assume $\delta<c$. Consider the function $r:=\frac12\log\frac{c}{u}$ so that 
$$\partial_t r=c^{-1}e^{2r}\Delta r.$$ 
It suffices to show that the non-negative part of $r$, denoted $r_+$, satisfies
$$ \lim_{t \downarrow 0} \sup_{z \in K} r_{+}(z,t) = 0.$$
We first choose a smooth approximation of $(\cdot)_+$. That is, fix a smooth non-decreasing function $\theta:\mathbb{R} \to [0,1]$ with $\theta \equiv 0$ on $(-\infty,0]$ and $\theta \equiv 1$ on $[1,\infty)$, and then for $\varepsilon>0$ define
$$\phi_\varepsilon(x) := \begin{cases}
\int_0^x \theta(s\varepsilon^{-1}) ds &: x > 0\\
0  &: x \leq 0
\end{cases}.$$ 
It follows that $\phi_\varepsilon$ is smooth, convex and non-decreasing, with $\phi_\varepsilon' \geq 0$. Moreover, we have the two-sided inequality
$$ 0 \leq (x)_+ - \phi_{\varepsilon}(x) \leq \varepsilon, \quad \forall x \in \mathbb{R},$$
and hence $\phi_\varepsilon \to (\cdot)_+$ uniformly. If we now set $w_\varepsilon:=\phi_\varepsilon(r)$, then by a direct computation
\begin{align*}
    \partial_t w_{\varepsilon} &=c^{-1} e^{2r} \phi_\varepsilon'(r)  \Delta r\\
    &= c^{-1} \nabla \cdot (e^{2r} \phi_{\varepsilon}'(r) \nabla r) -c^{-1} e^{2r} |\nabla r|^2 \left(2 \phi_{\varepsilon}'(r) + \phi_{\varepsilon}''(r) \right)\\
    &\leq c^{-1} \nabla \cdot (e^{2r} \phi_{\varepsilon}'(r) \nabla r) = c^{-1} \nabla \cdot (e^{2r} \nabla w_{\varepsilon}).
\end{align*}
We note that $\nabla w_{\varepsilon} = 0$ if $r \leq 0$, and so defining
$$ a(r) := \begin{cases}
    e^{2r} &: r > 0\\
    1 &: r \leq 0
\end{cases},$$
we have that
\begin{equation}
    \partial_t w_{\varepsilon} \leq c^{-1} \nabla \cdot (a(r) \nabla w_{\varepsilon}).
\end{equation}

Next we derive a local $L^2$-estimate for $w_\varepsilon$. Choose
\[
    K\Subset U_1\Subset U_2\Subset U_3\Subset\Omega,
\]
and let $\eta\in C_c^\infty(U_2)$ satisfy $\eta\equiv1$ on $U_1$. For $0<s<t<T$, multiplying the preceding differential inequality by
$\eta^2w_\varepsilon$, integrating by parts, and then integrating in
time gives
\begin{equation*}
    \frac{1}{2} \int\eta^2w_\varepsilon^2(\cdot,t) + c^{-1}\int_s^t\int a \eta^2|\nabla w_\varepsilon|^2 \leq \frac{1}{2} \int \eta^2 w^2_\varepsilon(\cdot,s) - 2c^{-1}\int_s^t\int a \eta w_\varepsilon \nabla w_\varepsilon\cdot\nabla\eta.
\end{equation*}
Since $1 \leq a \leq \frac{c}{\delta}$, we apply Young's inequality to yield the estimate
\begin{align*}
2c^{-1} a \eta |w_\varepsilon| |\nabla w_\varepsilon| |\nabla\eta| &\leq \frac{1}{2c} a \eta^2 |\nabla w_\varepsilon|^2 + \frac{2}{c} a w_\varepsilon |\nabla \eta|^2\\
&\leq \frac{1}{2c} a \eta^2 |\nabla w_\varepsilon|^2 + C w_\varepsilon |\nabla \eta|^2.
\end{align*}
This allows the absorption of the gradient term into the left hand side, and we find that
\begin{align*}
\int_{U_1}w^2_{\varepsilon}(\cdot,t) &\leq \int\eta^2w^2_{\varepsilon}(\cdot,s) +  C\int_s^t\int_{U_2}w_{\varepsilon}^2\\
&\leq \int\eta^2 r_+^2(\cdot,s) +  C\int_s^t\int_{U_2}w_{\varepsilon}^2.
\end{align*}

By assumption 
$$\limsup_{s \downarrow 0} r(z,s) \leq 0,$$
for almost every $z \in \Omega$. In particular, taking $s \downarrow 0$ in the above integral estimate and applying the dominated convergence theorem results in
$$ \int_{U_1}w^2_{\varepsilon}(\cdot,t)  \leq   C\int_0^t\int_{U_2}w_{\varepsilon}^2. $$

Repeating the same estimate with a cutoff supported in $U_3$ and
equal to one on $U_2$ gives
$$\int_{U_2}w^2_{\varepsilon}(\cdot,t)  \leq   C\int_0^t\int_{U_3}w_{\varepsilon}^2.$$
Combining the two inequalities with the estimate $0 \leq w_{\varepsilon} \leq \frac{1}{2} \log \frac{c}{\delta}$, we have the $L^2$-bound
$$\int_{U_1} w_{\varepsilon}^2(\cdot,t) \leq C\int_0^t\int_{U_2}w_{\varepsilon}^2(\cdot,\tau) \leq C\int_0^t \int_0^{\tau} \int_{U_3}w_{\varepsilon}^2 \leq Ct^2.$$

A standard local $L^2-L^\infty$ estimate for uniformly parabolic equations (see e.g. \cite[Chapter III, \S 8]{LSU1968} or \cite[Theorem 6.17]{Lie96}) gives, for $z \in K$ and $t>0$ sufficiently small
\begin{equation*}
 w_{\varepsilon}^2(z,t) \leq Ct^{-2} \int_{t/2}^{t} \int_{U_1} w_{\varepsilon}^2 \leq Ct, 
\end{equation*}
where the constant $C$ is independent of $\varepsilon$ as the coefficient $a$ was independent of $\varepsilon$. In particular
\begin{equation*}
\sup_{z \in K} r_+(z,t) \leq C \sqrt{t} + \varepsilon.
\end{equation*}
Finally, sending $\varepsilon\downarrow0$ proves the result.
\end{proof}

\begin{lem}\label{lem:njb}
Fix $s\in (0,T)$ and let $E\subseteq M_s^+\setminus M_s$ be compact. Then, for every $C>0$, there exist an open neighbourhood $E \subseteq \Omega \Subset M_s^+$ and $\varepsilon>0$ such that
$$v(z,t)\geq C, \quad \forall (z,t)\in \Omega\times(s,s+\varepsilon).$$
\end{lem}

\begin{proof}
By the Schwarz lemma, 
$$g(t)\geq 2t h_t \geq  2s \cdot h_D,$$
for every $t \in [s,T)$. By applying Lemma~\ref{lem uhlb} to each point $y\in E$, we find an open neighbourhood $E\subset \Omega' \Subset M_s^+$ on which $v(z,s)\geq C+1$ for every $z\in \Omega' \cap M_s$. Since $\mu_{\bar{g}}(M_s^+ \setminus M_s) = 0$
$$ \liminf_{t \downarrow s} v(z,t) \geq C+1, \quad \textnormal{for a.e. } z \in \Omega'.$$
As $\overline{\Omega'} \Subset D$, we also have a uniform lower bound $v(z,t)\geq \delta >0$ on $\Omega' \times(s,T)$. Therefore, the hypotheses of Lemma~\ref{lem:LLFDE} are satisfied and we may choose an open neighbourhood $E \subset \Omega \Subset \Omega'$, such that
$$\liminf_{t\downarrow s}\inf_{z \in \Omega} v(z,t)\geq C+1.$$
Choosing $\varepsilon > 0$ sufficiently small, the result follows.
\end{proof}

\begin{proof}[Proof of Theorem~\ref{thm:cf pp}]
Fix $J\Subset(0,T)$ and $A<\infty$. It suffices to show that the set
$$K :=
\{(z,t)\in\mathcal M:t\in J,\ v(z,t)\le A\},$$
is compact in $\mathcal M$. Suppose not, and that there exists a sequence $(z_i,t_i)\in K$ leaving every compact subset of $\mathcal M$. After passing to a subsequence, $t_i \to s \in J$. By
Lemma~\ref{lem: VD}, we may further assume that
$$t_i \downarrow s, \quad z_i \to z_\infty\in M_s^+\setminus M_s.$$

By Theorem~\ref{thm:fld main}, $M_s^+\setminus M_s$ is closed in $M_s^+$ and $\mu_{\bar{g}}(M_s^+ \setminus M_s) = 0$. Choose a relatively compact neighbourhood $U\Subset M_s^+$ of $z_\infty$. Applying
Lemma~\ref{lem:njb} to
$$E=(M_s^+\setminus M_s)\cap\overline U$$
with any constant larger than $A$, we obtain $v(z_i,t_i)>A$
for all sufficiently large $i$, which is a contradiction.
\end{proof}

We begin with the usual comparison principle. Recall, if $\Omega\subset D\times(0,T)$ is open, its parabolic boundary is 
$$ \partial_P \Omega := \left\{ (z,t)\in\partial\Omega: B_r(z)\times(t-r^2,t) \not\subset \Omega,  \text{ for any } r > 0 \right\}.$$

\begin{lem}\label{comp 1}
Let $\Omega\Subset D\times (0,T)$ be a relatively compact open subset, and let
\[
    u,v\in C^{2,1}(\overline\Omega)
\]
be positive solutions of the LFDE 
\begin{equation}\label{eq:LFDE}
    \partial_t u = \Delta \log u.
\end{equation}
If $u\leq v$ on the parabolic boundary $\partial_P\Omega$, then $u\leq v$ on all of $\Omega$.
\end{lem}

\begin{proof}
Let $w:=u-v$ and define
$$a(z,t) := \int_0^1 \frac{d\theta}
    {\theta u(z,t)+(1-\theta)v(z,t)} > 0.$$
Then $\log u-\log v=aw$, and
$$\partial_t w = \Delta(aw) = a\Delta w+2\nabla a\cdot\nabla w+(\Delta a)w.$$
Since $\overline\Omega$ is compact and $u,v$ are positive and smooth,
all the coefficients in this equation are bounded, with $a$ bounded
away from zero. Choose $c>\sup_{\overline\Omega}\Delta a$ and set $\widetilde w=e^{-ct}w$. Then
$$\partial_t\widetilde w = a\Delta\widetilde w +2\nabla a\cdot\nabla\widetilde w +(\Delta a-c)\widetilde w.$$
The weak parabolic maximum principle on the open spacetime domain $\Omega$ therefore gives \[ \sup_{\overline\Omega}\widetilde w \leq \max\left\{ 0,\, \sup_{\partial_P\Omega}\widetilde w \right\}. \] Since $\widetilde w\leq0$ on $\partial_P\Omega$, we conclude that $\widetilde w\leq0$ throughout $\Omega$. Hence $u\leq v$.
\end{proof}

\begin{lem}\label{lem:comp main}
Suppose $\mathcal M\subset D\times (0,T)$ is a complete Ricci flow spacetime and write
$$g(t)=v(z,t)|dz|^2, \quad \forall (z,t) \in \mathcal{M}.$$
Suppose that $u:D \times (0,T) \to (0,\infty)$ is a solution to the LFDE which is uniformly bounded on compact time intervals. If for some $s\in (0,T)$ we have the ordering
$$u(z,s)\leq v(z,s) \quad \forall z \in M_s,$$
then
$$u(z,t) \leq v(z,t)
    \quad \forall z \in M_t,$$
for every $t\in (s,T)$.
\end{lem}

\begin{proof}
Fix $\tau \in (s,T)$ and let $c:=\sup_{D\times[s,\tau]}u<\infty$. The open subset
$$\Omega := \left\{(z,t)\in\mathcal M: t \in (s,\tau),\quad v(z,t) <c \right\},$$
is relatively compact in $\mathcal{M}$ by the parabolic properness of $v$ as shown in Theorem~\ref{thm:cf pp}. Moreover, we have $u \leq v$ on $\partial_P \Omega$. Hence the usual comparison principle gives
$$u(z,t) \leq v(z,t), \quad \forall (z,t) \in \Omega.$$
On the complement of $\Omega$ in $\mathcal M \cap D \times (s,\tau)$, we also have $u \leq c \leq v$. Consequently,
$$u(z,t) \leq v(z,t), \quad \forall (z,t) \in \mathcal{M} \cap (D \times (s,\tau)).$$
Since $\tau$ was arbitrary, the result follows.
\end{proof}

\section{Ricci Flows from Measures}

The comparison theorem allows the use of Ricci flows as barriers to our spacetime solution. The barriers we will use arise by evolving Radon measures under the Ricci flow. We will use the following existence and uniqueness results of Topping and Yin \cite{topping2024smoothing,topping2024uniqueness}. 

For a conformal surface $M$ and a Radon measure $\mu$ on $M$, we write $\widetilde{M}$ for the universal cover of $M$, and $\widetilde{\mu}$ for the lift of $\mu$ to $\widetilde M$.

\begin{thm}\label{thm: smoothing measures}
Let $M^2$ be a smooth surface equipped with a conformal structure, and let $\mu$ be a Radon measure on $M$ which is non-atomic, in the sense that 
$$ \mu(\{x\}) = 0, \quad \forall x \in M.$$
Define $T \in [0,\infty]$ by
\begin{itemize}
    \item $T = \infty$ if $\widetilde M = D$, the unit disk in the plane;
    \item $T = \frac{\widetilde \mu (\widetilde M)}{4 \pi}$ if $\widetilde M = \mathbb{C}$;
    \item $T = \frac{\widetilde \mu (\widetilde M)}{8 \pi}$ if $\widetilde M = S^2$.
\end{itemize}
If $T \in (0,\infty]$, then there exists a smooth complete conformal Ricci flow $g(t)$ on $M$, for $t \in (0,T)$, attaining $\mu$ as initial data in the sense that
$$ \mu_{g(t)} \rightharpoonup \mu \quad \textnormal{as } t \downarrow 0,$$
where $\mu_{g(t)}$ denotes the volume measure of the metric $g(t)$. Moreover, whatever the value of $T \in [0,\infty]$, if $\widetilde g(t)$, $t \in (0,\widetilde T)$, is any smooth complete conformal
Ricci flow on $M$ that attains $\mu$ as initial data in the same sense, then $\widetilde T \leq T$ and
$$ \widetilde g (t) \equiv g(t), \quad \forall t \in (0,\widetilde T).$$

\end{thm}

The following lemma prevents area from decreasing too rapidly along such a Ricci flow (cf. \cite[Lemma 2.2]{peachey2025all}).

\begin{lem}\label{lem:lower vol}
Let $g(t)$ be the complete conformal Ricci flow on $M \times (0,T)$ starting weakly from $\mu$ as in Theorem~\ref{thm: smoothing measures}. Fix $0<r<R<\infty$ and suppose $B_r \subset B_R \Subset M$ are concentric balls in some local complex coordinate chart on $M$. Then
\begin{equation*}
\mu(B_r) \leq  \mu_{g(t)}(B_R) + \frac{8\pi t}{1-(r/R)^2}, \quad \forall t \in \left(0, T \wedge \frac{\mu(B_r)}{8\pi} \right).
\end{equation*}
\end{lem}

\begin{proof}
Let $\pi:\widetilde M\to M$ be the universal cover. Choose a connected component $\widetilde B_R$ of $\pi^{-1}(B_R)$. Since $B_R$ is simply connected, the restriction $ \pi:\widetilde B_R \to B_R$ is a diffeomorphism. Let $\widetilde B_r\subset\widetilde B_R$ denote the corresponding lift of $B_r$. If we set $\widetilde g(t):=\pi^*g(t)$ and let $\widetilde\mu$ denote its weak initial measure, then $\widetilde g(t)$ is complete and
$$\widetilde\mu(\widetilde B_r)=\mu(B_r),
\quad
\mu_{\widetilde g(t)}(\widetilde B_R)
=
\mu_{g(t)}(B_R).$$

If $\mu(B_r)=0$ there is nothing to prove, so suppose
$\mu(B_r)>0$. By uniformisation, $\widetilde M$ is conformally
equivalent to $S^2$, $\mathbb C$, or $D$. In the latter two cases we may regard $\widetilde M$ as an open subset of
$S^2\cong\mathbb C\cup\{\infty\}$.

We now take $\nu:=\widetilde\mu \llcorner_{\widetilde B_r}$ to be a finite non-atomic Radon measure on $S^2$. Let $h(t)$ denote the complete Ricci flow on $S^2$ starting weakly
from $\nu$. Since $\nu(S^2) = \mu(B_r)$, this flow exists for
$t \in (0,\frac{\mu(B_r)}{8\pi})$. By the comparison theorem for Ricci flows on a fixed surface \cite[Theorem 1.2]{topping2024uniqueness}, the inequality of initial measures
$\nu \leq \widetilde\mu$ on $\widetilde M$ implies $h(t) \leq \widetilde g(t)$ on $\widetilde M$. Consequently
$$ \mu_{h(t)}(\widetilde B_R) \leq \mu_{\widetilde g(t)}(\widetilde B_R), \quad \forall t \in \left(0, T \wedge \frac{\mu(B_r)}{8\pi} \right).$$

Applying the estimate for Ricci flows on the sphere \cite[Lemma 2.2]{peachey2025all} to $h(t)$ gives
$$\nu(\widetilde B_r)
\leq
\mu_{h(t)}(\widetilde B_R)
+
\frac{8\pi t}{1-(r/R)^2}, \quad \forall t \in \left(0, T \wedge \frac{\mu(B_r)}{8\pi} \right).$$
Combining these two inequalities with $\nu(\widetilde B_r) = \widetilde \mu( \widetilde B_r) = \mu(B_r)$, and $\mu_{\widetilde g(t)}(\widetilde B_R) = \mu_{g(t)}(B_R)$ yields the result.
\end{proof}

Let $g(t)$, $t \in (0,T)$ be a complete conformal Ricci flow on $M$ attaining some non-atomic Radon measure $\mu$ as initial data, as in Theorem~\ref{thm: smoothing measures}. For $\lambda>0$, consider the parabolically rescaled Ricci flow
$$g_\lambda(t):=\lambda g(t/\lambda), \quad t\in(0,\lambda T).$$
Recall that for complete Ricci flows we have the scalar curvature lower bound $\mathcal R_{g(t)}\ge-\frac1t$, from which it follows that
$$\frac{\partial}{\partial t}
\left(\frac{g(t)}{2t}\right) = -(1+t\mathcal R_{g(t)}) \frac{g(t)}{2t^2} \leq 0.$$
Consequently, for every fixed $t>0$,
$$g_{\lambda_1}(t) =
2t \frac{g(t/\lambda_1)}{2(t/\lambda_1)} \leq 2t \frac{g(t/\lambda_2)}{2(t/\lambda_2)} = g_{\lambda_2}(t), \quad (\lambda_1 \leq \lambda_2).$$
Thus the parabolic rescalings are non-decreasing in $\lambda$. One may therefore consider the (possibly degenerate) limit of $g_{\lambda}(t)$ as $\lambda \uparrow \infty$.

\begin{lem}\label{lem:pbu}
Let $g(t)$, for $t \in (0,T)$, be a complete conformal Ricci flow on $M$ attaining the non-atomic Radon measure $\mu$ as initial data in the sense of Theorem~\ref{thm: smoothing measures}. Then the following hold
\begin{enumerate}[label=(\roman*)]
    \item There exists a smooth Ricci flow $g_{\infty}(t)$ on $M \setminus \supp(\mu)$, for $t \in (0,\infty)$, such that $g_{\lambda}(t)$ converges to $g_{\infty}(t)$ in a smooth local sense as $\lambda \uparrow \infty$.
    \item $g_\infty(t) = t \cdot g_{\infty}(1)$, with $\mathcal{R}_{g_\infty(1)} \equiv -1$.
    \item For any open set $\Omega \subseteq M$ with $\Omega \cap \supp(\mu) \neq \emptyset$ we have
    $$ \lim_{\lambda \uparrow \infty} \mu_{g_{\lambda}(t)} (\Omega) = \infty, \quad \forall t \in (0,\infty).$$
\end{enumerate}
\end{lem}

\begin{proof}
Fix $p\in M\setminus\supp(\mu)$ and choose a local complex coordinate centred at $p$, in which $B_{3r}\Subset M\setminus\supp(\mu)$ for some $r>0$. Writing $g(t)=u(z,t)|dz|^2$ we have
$$g_\lambda(t)=u_\lambda(z,t)|dz|^2, \quad u_\lambda(z,t)=\lambda u(z,t/\lambda).$$
Since $\mu(\overline{B_{2r}})=0$, the $L^1-L^\infty$ smoothing from Theorem~\ref{thm l1-linfty} gives
$$\sup_{z \in B_r}u(z,t) \leq C_0r^{-2}t, \quad t\in(0,T).$$
Consequently
$$\sup_{z \in B_r}u_\lambda(z,t)
\leq C_0r^{-2}t,$$
whenever $t \in (0,\lambda T)$. On the other hand,
$$u_\lambda(z,t) = t \cdot \frac{u(z,t/\lambda)}{(t/\lambda)},$$
is non-decreasing in $\lambda$. Thus, on every compact subset of
$(M\setminus\supp(\mu))\times(0,\infty)$, the conformal factors
$u_\lambda$ admit uniform positive lower and upper bounds for all
sufficiently large $\lambda$. Standard interior parabolic estimates
therefore give uniform local $C^k$-bounds for every $k$. Finally, by Arzel{\`a}-Ascoli, our sequence of Ricci flows $g_\lambda(t)$ converges locally smoothly to some limiting Ricci flow $g_\infty(t)$ on $(M \setminus \supp{\mu}) \times (0,\infty)$.

For any fixed $t>0$, we note that
\begin{equation*}
    g_\infty(t) := \lim_{\lambda \uparrow \infty} g_\lambda(t) = t \cdot \lim_{\lambda \uparrow \infty} g_{\lambda/t} (1)= t \cdot g_{\infty}(1).
\end{equation*}
Substituting this into the Ricci flow equation
\begin{equation*}
g_{\infty}(1) = \frac{\partial g_\infty}{\partial t} = -2\Ric(g_\infty) = - \mathcal{R}_{g_{\infty}(1)} \cdot g_{\infty}(1),
\end{equation*}
yields $\mathcal{R}_{g_{\infty}(1)} \equiv -1$. Finally, choose $p \in \Omega \cap \supp(\mu)$. In some local coordinate chart on $M$ centred at $p$, as $\Omega$ is open we have $B_{2r} \subseteq \Omega$ for $r>0$ sufficiently small. Applying Lemma~\ref{lem:lower vol} we find
\begin{equation*}
\mu_{g(t)}(B_{2r}) \geq \mu(B_r) - 20\pi t, \quad \forall t \in \left(0, T \wedge \frac{\mu(B_r)}{8\pi} \right).
\end{equation*}
In particular, for any $t \in (0,\infty)$, we may choose $\lambda$ sufficiently large so that $\frac{t}{\lambda} \in \left(0, T \wedge \frac{\mu(B_r)}{40\pi} \right)$, and hence
\begin{equation*}
 \mu_{g_\lambda(t)}(\Omega) \geq    \mu_{g_\lambda(t)}(B_{2r}) = \lambda \cdot \mu_{g\left(\frac{t}{\lambda}\right)}(B_{2r}) \geq \lambda \cdot \left( \mu(B_r) - \frac{20\pi t}{\lambda} \right) \geq \frac{\lambda}{2} \cdot  \mu(B_r).
\end{equation*}
Taking $\lambda \uparrow \infty$ yields the result.
\end{proof}

\section{Perfect Barriers}

\begin{lem}\label{lem:perfect int}
If $P \subseteq M$ is a perfect subset, $O \subseteq M$ open and $P \cap O \neq \emptyset$, then there exists a non-empty compact perfect set $P' \subseteq P \cap O$.
\end{lem}

\begin{proof}
Fix $p \in P \cap O$, choose an open neighbourhood $p \in V \Subset O$, and let $U:= V \cap P$. Then $U$ is a relatively open neighbourhood of $p$ in $P$, and its closure in $P$ is contained in $O$. Let $P'$ be precisely this closure in $P$, so that $P'$ is a non-empty compact subset of $P \cap O$. It remains to show that $P'$ has no isolated points. 

Suppose to the contrary that $x\in P'$ is isolated. Then there exists a relatively open neighbourhood $W$ in $P$ of $x$, with $W\cap P'= \{x\}$. As $x$ is in the closure of $U$ in $P$, we have $W\cap U$ is non-empty. But $U\subset P'$, and so $W\cap U=\{x\}$. Since both $U$ and $W$ are relatively open in $P$, the singleton $\{x\}=W\cap U$ is relatively open in $P$, contradicting the fact that $P$ is perfect. Thus $P'$ has no isolated points and is closed, hence perfect.
\end{proof}

\begin{lem}[{\cite[Corollary 2.8]{hebert1968supports}}]\label{HebertLacey}
If $P$ is a non-empty compact and perfect subset of $M$, then there exists a non-atomic Radon measure $\mu$ on $M$ with $\supp(\mu) = P$.
\end{lem}

\begin{proof}[Proof of Theorem~\ref{thm scattered}]
Fix $s\in I$. We prove that $M\setminus M_s=\{x\in M:t_-(x)\ge s\}$ is scattered.

By the Cantor-Bendixson theorem (see e.g. \cite{kechris1995perfect}), we may decompose
$$ M \setminus M_s = X \sqcup P,$$
into a perfect set $P$ and a scattered set $X$. We assume for a contradiction that $P$ is non-empty. 

Choose $T \in (s,\beta)$ such that $P \cap M_{T} \neq \emptyset$, and consider the collection of times between $s$ and $T$ for which the spatial slice is hyperbolic
$$ \mathcal{H} := \{ t \in [s,T] : M_t \textnormal{ is hyperbolic} \}.$$
By our assumption that $P$ is a non-empty perfect set, Lemma~\ref{lem uni planar} implies that $s \in \mathcal{H}$. Let $\tau := \sup \mathcal{H}$. As a subset of a hyperbolic surface is itself hyperbolic, it follows that $M_t$ is hyperbolic for every $t \in (\alpha,\tau)$. 

\begin{claim}
$P \cap M_{\tau}^+ \neq \emptyset$.
\end{claim}

\begin{proof}[Proof of Claim]
If $\tau = T$, then $P \cap M_{\tau}^+ \supseteq P \cap M_{\tau} \neq \emptyset$, and so we may restrict to the situation $\tau < T$. In this case, for any sequence $t_j \in (\tau,T)$ with $t_j \downarrow \tau$, $M_{t_j}$ is non-hyperbolic and by Lemma~\ref{lem uni planar}, $M \setminus M_{t_j}$ consists of at most two points. As
$$ \bigcap_{j} M_{t_j} = \bigcap_{t>\tau} M_t,$$
and the complement in $M$ consists of at most two points, this intersection is open and hence
$$ M_{\tau}^+ = \bigcap_{t > \tau} M_t.$$
Now, since $P$ is infinite and $P \setminus M_{\tau}^+$ consists of at most two points, $P \cap M_{\tau}^+ \neq \emptyset$.
\end{proof}

Fix $t_0 \in (\alpha,s)$. After translating we may assume that $t_0=0$.  The interaction between $P$ and $M_{\tau}^+$ may now occur in two essentially different ways: either some of it is already present at time $\tau$, in which case the hyperbolic comparison principle applies on an earlier lifted slice, or it lies entirely in $M_\tau^+\setminus M_\tau$, in which case the inception measure at $\tau$ supplies the required upper-volume bound. We treat these two distinct cases separately.

\begin{enumerate}
\item Suppose $P \cap M_{\tau} \neq \emptyset$. In this case we may find $\tau' \in (s,\tau)$ such that $P \cap M_{\tau'} \neq \emptyset$. By Lemma~\ref{lem:perfect int} there exists a compact non-empty perfect set $P' \subseteq P \cap M_{\tau'}$.

As $M_{\tau'}$ is hyperbolic, we may lift the spacetime to the disk $\mathcal{M} \subset D \times (0,\tau')$. As the lift under the universal cover of a perfect set is perfect, the lift of $P'$, denoted $\tilde{P}'$, is perfect. Applying Lemma~\ref{lem:perfect int} again, there exists a compact non-empty perfect set $P'' \subseteq \tilde{P}'$.

By Lemma~\ref{HebertLacey}, there exists a non-atomic non-zero Radon measure $\mu$ with $\supp(\mu) = P''$. Extend $\mu$ by zero to $\mathbb{C}$, and for any $\lambda > 0$, let $G_{\lambda}(t)$ denote the complete Ricci flow on $\mathbb{C}$ starting weakly from $\lambda \cdot \mu$. For $\lambda$ sufficiently large, it follows that $G_{\lambda}(t)$ exists for all $t \in (0,\tau']$. 

By the monotonicity of $t \mapsto \frac{G_{\lambda}(t)}{t}$ and the Schwarz lemma, we have the ordering
$$ G_{\lambda}(s) \leq  s \cdot G_{\infty}(1) \leq 2s \cdot h_{s} \leq g(s) \quad \textnormal{on } M_{s}. $$
Since $G_{\lambda}(t)$ is smooth on $\overline{D} \times (0,\tau')$, it follows that its conformal factor is bounded on compact subsets of time. We may therefore apply the comparison principle from Lemma~\ref{lem:comp main} to find
$$ G_{\lambda}(t)  \leq g(t) \quad \textnormal{on } M_{t},$$
for every $t \in (s,\tau')$. In particular, it follows that for every $B \Subset D$ we have
\begin{align*}
 \mu_{G_{\lambda}(\tau')}(B) =\lim_{t \uparrow \tau'} \mu_{G_{\lambda}(t)}(B) \leq \lim_{t \uparrow \tau'} \mu_{g(t)}(B) = \mu_{g(\tau')}(B) < \infty.
\end{align*}
Conversely, suppose $B \Subset D$ is any open set containing a point in $P''$. Then it follows from Lemma~\ref{lem:pbu} that
\begin{equation*}
    \lim_{\lambda \uparrow \infty} \mu_{G_\lambda(\tau')}(B) = \infty,
\end{equation*}
a contradiction.

\item Suppose instead that $P \cap M_{\tau} = \emptyset$. By Lemma~\ref{lem:perfect int}, there exists a compact non-empty perfect set $P' \subseteq P \cap M_{\tau}^+$. In particular, $P' \subseteq M_{\tau}^+ \setminus M_{\tau}$.

By Lemma~\ref{HebertLacey}, there exists a non-atomic non-zero Radon measure $\mu$ with $\supp(\mu) = P'$. For any $\lambda > 0$, let $G_{\lambda}(t)$ denote the complete Ricci flow on $M_{\tau}^+$ starting weakly from $\lambda \cdot \mu$. For $\lambda$ sufficiently large, it follows that $G_{\lambda}(t)$ exists for all $t \in (0,\tau]$.

As $M_{\tau} \subseteq M_{\tau}^+ \setminus P'$, the monotonicity of $t \mapsto \frac{G_{\lambda}(t)}{t}$ together with the Schwarz lemma gives
$$ G_{\lambda}(t) \leq t \cdot G_{\infty}(1) \leq 2t \cdot h_{t} \leq g(t) \quad \textnormal{on } M_{t},$$
for every $t \in (0,\tau)$. By continuity, taking $t \uparrow \tau$, we have $G_{\lambda}(\tau) \leq g(\tau)$ on $M_{\tau}$.

Now, for any $B \Subset M_{\tau}^+$, Theorem~\ref{thm:fld main} and Theorem~\ref{thm:inception} tell us that
\begin{align*}
\mu_{G_{\lambda}(\tau)}(B) = \mu_{G_{\lambda}(\tau)}(B \cap M_{\tau}) \leq \mu_{g(\tau)}(B \cap M_{\tau}) = \mu_{\tau}^+(B \cap M_{\tau}) \leq \mu_{\tau}^+(B) < \infty.
\end{align*}
Conversely, as in the previous case, suppose $B \Subset M_{\tau}^+$ is any open set containing a point in $P'$. Then it follows from Lemma~\ref{lem:pbu} that
\begin{equation*}
    \lim_{\lambda \uparrow \infty} \mu_{G_\lambda(\tau)}(B) = \infty,
\end{equation*}
a contradiction.
\end{enumerate}
Therefore, our original assumption must have been false, $P = \emptyset$ and $M \setminus M_s$ is a scattered subset of $M$. In particular, because $M$ is second countable, every scattered subspace of $M$ is countable, and choosing $s_j\downarrow\alpha$,
$$ \{x:\mathfrak t_-(x)>\alpha\} = \bigcup_j\{x:\mathfrak t_-(x)\ge s_j\}, $$
is countable. This concludes the proof of Theorem~\ref{thm scattered} and hence of Theorem~\ref{thm:main}.
\end{proof}

\section{Completeness of Soliton Vector Fields}

Let $(M,g,X)$ be a $2$-dimensional Ricci soliton; that is, $(M,g)$ is a Riemannian surface and $X$ a smooth vector field on $M$ such that
\begin{equation*}
    2\Ric(g) + \varepsilon g = \mathcal{L}_Xg,
\end{equation*}
for some $\varepsilon \in \mathbb{R}$. After rescaling the soliton, we may always assume $\varepsilon \in \{ 0 , \pm1 \}$. In the following, we consider the expanding case $(\varepsilon = 1)$ for ease of exposition; the shrinking and steady cases are analogous and will be dealt with at the end.

Recall first the usual picture when the soliton vector field $X$ is complete. If $\phi_s : M \to M$ denotes its flow, normalised by $\phi_0 = \operatorname{id}_M$, it follows that 
$$g(t) := t \phi_{-\log t}^*(g), \quad t \in (0,\infty),$$
is a Ricci flow. This solution is self-similar in the sense that, for every $\lambda > 0$,
$$\phi_{\log \lambda}^*  g(t)  = \lambda g \left(t/\lambda\right).$$

Thus the associated Ricci flow is invariant under parabolic rescaling. For our purposes, it is useful to regard this scaling symmetry, rather than the global flow of $X$, as the fundamental structure.

\begin{defn}
An expanding Ricci soliton spacetime is a Ricci flow spacetime $(\mathcal M^{n+1},\mathfrak t,\partial_t,G)$ over $(0,\infty)$, equipped with a smooth multiplicative action
$$ \Phi:(0,\infty)\times\mathcal M\to\mathcal M,$$
such that
$$\mathfrak t\circ\Phi_{\lambda}= \lambda \mathfrak t, \quad
(\Phi_\lambda)_* \partial_t = \lambda \partial_t, \quad \Phi_\lambda^*G = \lambda G.$$
\end{defn}

The advantage of this formulation is that the corresponding spacetime can be constructed directly from the soliton equation without assuming that $X$ is complete. The vector field $X$ is incorporated into the spacetime time vector field, while dilation of the time coordinate supplies the required scaling action.

\begin{lem}\label{lem: sol st}
Every expanding Ricci soliton $(M,g,X)$ induces an expanding Ricci soliton spacetime. Moreover, if $g$ is complete, then the induced Ricci soliton spacetime is complete. 
\end{lem} 

\begin{proof}
Define 
$$\mathcal{M} = M \times (0,\infty), \quad \mathfrak{t}(x,t) = t, \quad \partial_t := \frac{\partial}{\partial t} - \frac{1}{t} X, \quad G = t \cdot g.$$
It follows that $(\mathcal{M},\mathfrak{t},\partial_t,G)$ is a well-defined Ricci flow spacetime over $(0,\infty)$. Indeed
$$ \partial_t \mathfrak{t} = \left(\frac{\partial}{\partial t} - \frac{1}{t} X \right) \cdot t \equiv 1, $$
and
\begin{align*}
\mathcal{L}_{\partial_t} G + 2\Ric(G) = g - \mathcal{L}_X(g) + 2\Ric(g) = 0.
\end{align*}

Define the smooth multiplicative action 
$$\Phi : (0,\infty) \times \mathcal{M} \to \mathcal{M}, \quad \Phi_\lambda(x,t):=(x,\lambda t).$$ 
It follows that $\mathfrak{t}\circ\Phi_\lambda = \lambda\mathfrak{t}$, $\Phi_\lambda^*G = \lambda G$, and 
\begin{align*} 
(\Phi_\lambda)_*  \partial_t = \lambda\frac{\partial}{\partial t}- \frac{1}{t}X = \lambda \left( \frac{\partial}{\partial t} - \frac{1}{\lambda t}X \right) = \lambda \partial_t. 
\end{align*}

We have shown that $(\mathcal{M},\mathfrak{t},\partial_t,G)$ is an expanding Ricci soliton spacetime. If $g$ is complete, then for every $t>0$, $(M,G(t)=t \cdot g)$ is complete. Hence the resulting Ricci soliton spacetime is complete. 
\end{proof}

The scaling action descends to a multiplicative action on $\mathcal{W}$.

\begin{lem}\label{lem:soliton-worldlines}
Let \((\mathcal M^{n+1},\mathfrak t,\partial_t,G)\) be an expanding Ricci soliton spacetime. Then the multiplicative action induces an action on $\mathcal W$, still denoted by $\Phi_\lambda$, satisfying
$$ \mathfrak t_-(\Phi_\lambda(w)) = \lambda \mathfrak t_-(w),
\quad \mathfrak t_+(\Phi_\lambda(w)) = \lambda \mathfrak t_+(w),
$$
for every $w\in\mathcal W$ and $\lambda>0$.
\end{lem}

\begin{proof}
Since $(\Phi_\lambda)_*\partial_t=\lambda\partial_t$, each $\Phi_\lambda$ maps integral curves of $\partial_t$ to integral curves of $\partial_t$. As $\Phi_\lambda^{-1}=\Phi_{1/\lambda}$, maximal worldlines are mapped to maximal worldlines, and hence the action induces a map on $\mathcal W$. Since
$$ \mathfrak t(\Phi_\lambda(x)) = \lambda\mathfrak t(x),$$
it follows that
$$ \mathfrak t_-(\Phi_\lambda(w)) = \inf_{x\in q^{-1}(w)}\mathfrak t(\Phi_\lambda(x)) = \lambda\inf_{x\in q^{-1}(w)}\mathfrak t(x) = \lambda\mathfrak t_-(w).
$$
The corresponding identity for \(\mathfrak t_+\) is identical.
\end{proof}

\begin{lem}\label{lem:soliton}
Let \((\mathcal M^{2+1},\mathfrak t,\partial_t,G)\) be a complete expanding Ricci soliton spacetime. Then
$$\mathfrak t(q^{-1}(w))=(0,\infty), \quad \forall w\in\mathcal W.$$
\end{lem}

\begin{proof}
By Theorem~\ref{thm:main}, $\mathfrak{t}_+ \equiv \infty$ and 
$$B:=\{ w \in \mathcal{W} : \mathfrak{t}_-(w) > 0\},$$
is countable. Suppose $B$ is non-empty and choose $w \in B$. By Lemma~\ref{lem:soliton-worldlines},
$$\mathfrak t_-(\Phi_\lambda(w)) = \lambda\mathfrak t_-(w), \quad \forall \lambda > 0.$$
Since $\mathfrak t_-(w)>0$, these inception times are distinct for distinct values of $\lambda$, and so
$$\{\Phi_\lambda(w):\lambda>0\}\subseteq B$$
is an uncountable collection of distinct worldlines, contradicting the countability of $B$. Hence $B=\emptyset$, and therefore $\mathfrak t_- \equiv 0$.
\end{proof}

\sol*

\begin{proof}[Proof of Corollary~\ref{cor:sol}]
Let $(M,g,X)$ be an expanding Ricci soliton with $g$ complete. By Lemma~\ref{lem: sol st} it induces a complete expanding Ricci soliton spacetime. Fix $x \in M$ and let $w$ be the worldline through $(x,1) \in \mathcal{M}$. By Lemma~\ref{lem:soliton}, $\mathfrak{t}(q^{-1}(w)) = (0,\infty)$, and so there exists a smooth curve $\gamma: (0,\infty) \to M$ such that the worldline $q^{-1}(w)$ consists of points $(\gamma(t),t)$, for all $t \in (0,\infty)$. In particular, it follows that
$$ \gamma'(t) = -\frac{1}{t} X_{\gamma(t)}, \quad \gamma(1) = x.$$
Note that $\eta: \mathbb{R} \to M$, defined by $\eta(s) := \gamma(e^{-s})$ satisfies $\eta(0) = x$ and
$$ \eta'(s) = -e^{-s} \gamma'(e^{-s}) = -e^{-s}\left( -\frac{1}{e^{-s}} X_{\gamma(e^{-s})}\right) = X_{\eta(s)}.$$
Therefore $\eta$ is an integral curve of $X$ through $x$ defined for all $s \in \mathbb{R}$. Because $x$ was arbitrary, it follows that $X$ is complete.

Finally, let us deal with the shrinking and steady cases. The argument is the same: the associated complete Ricci flow spacetime carries a symmetry which takes a point with finite inception time and generates an uncountable set of distinct points with finite inception time, contradicting Theorem~\ref{thm:main}. For a shrinking soliton, normalised by $\varepsilon=-1$, consider
$$\mathcal M=M\times(-\infty,0),\quad G(t)=(-t)\cdot g,\quad \partial_t=\frac{\partial}{\partial t}+\frac1tX.$$
The dilations $\Phi_\lambda(x,t)=(x,\lambda t)$ preserve the spacetime structure and satisfy
$$\mathfrak t_-(\Phi_\lambda(w))=\lambda\mathfrak t_-(w).$$
Hence no worldline can have finite inception time, and every worldline is defined on $(-\infty,0)$. If $(\gamma(t),t)$ is the worldline through $(x,-1)$, then $\eta(s):=\gamma(-e^s)$ is a complete integral curve of $X$ through $x$. For a steady soliton, consider
$$\mathcal M=M\times\mathbb R,\quad G(t)=g,\quad \partial_t=\frac{\partial}{\partial t}-X.$$
Now the translations $\Phi_\lambda(x,t)=(x,t+\lambda)$ preserve the spacetime structure and give
$$\mathfrak t_-(\Phi_\lambda(w))=\mathfrak t_-(w)+\lambda.$$
Thus every worldline is defined on $\mathbb R$, and if $(\gamma(t),t)$ is the worldline through $(x,0)$ then $\eta(s):=\gamma(-s)$ is a complete integral curve of $X$ through $x$. This proves the result in all three cases.
\end{proof}

\printbibliography
\end{document}